\documentclass[11pt]{article}

\usepackage{mathpazo}   % With old-style figures and real smallcaps.
\usepackage{tikz}

\usepackage{nameref}
\usepackage{mathtools}
\usepackage{empheq}
\usepackage{comment}
\usepackage[shortlabels,inline]{enumitem}
\setlist[enumerate]{nosep}
\usepackage[colorlinks=true,
linkcolor=refkey,
urlcolor=lblue,
citecolor=red]{hyperref}
\usepackage[doc,wmm,hhb]{optional}
\usepackage{xcolor}

\usepackage{float}
\usepackage{soul}
\usepackage{graphicx}
\definecolor{labelkey}{rgb}{0,0.08,0.45}
\definecolor{refkey}{rgb}{0,0.6,0.0}
\definecolor{Brown}{rgb}{0.45,0.0,0.05}
\definecolor{lime}{rgb}{0.00,0.8,0.0}
\definecolor{lblue}{rgb}{0.7,0.7,0.99}
\definecolor{OliveGreen}{rgb}{0,0.6,0}
\definecolor{tyrianpurple}{rgb}{0.4, 0.01, 0.24}
\colorlet{hlcyan}{cyan!30}

\usepackage{stmaryrd}
\usepackage{amssymb}

\makeatletter
\def\namedlabel#1#2{\begingroup
   \def\@currentlabel{#2}%
   \label{#1}\endgroup
}
\makeatother

\newcommand{\seppthree}{\setlength{\itemsep}{-3pt}}

\usepackage[margin=0.92in,footskip=0.25in]{geometry}
\newcommand{\J}[1]{\ensuremath{J_{#1}}}
\newcommand{\R}[1]{\ensuremath{R_{#1}}}

\newcommand{\bx}{\ensuremath{\mathbf{x}}}

\newcommand{\weakly}{\ensuremath{\:{\rightharpoonup}\:}}

\newcommand{\nnn}{\ensuremath{{n\in{\mathbb N}}}}
\newcommand{\kkk}{\ensuremath{{k\in{\mathbb N}}}}
\newcommand{\thalb}{\ensuremath{\tfrac{1}{2}}}
\newcommand{\menge}[2]{\big\{{#1}~\big |~{#2}\big\}}

\newcommand{\To}{\ensuremath{\rightrightarrows}}

\newcommand{\fenv}[1]%
{\ensuremath{\,\overrightarrow{\operatorname{env}}_{#1}}}
\newcommand{\benv}[1]%
{\ensuremath{\,\overleftarrow{\operatorname{env}}_{#1}}}

\newcommand{\scal}[2]{\left\langle{#1},{#2}  \right\rangle}

\newcommand{\Tt}{\ensuremath{\widetilde{T}}}
\newcommand{\tr}{\ensuremath{\triangleright}}

\newcommand{\RR}{\ensuremath{\mathbb R}}

\newcommand{\dom}{\ensuremath{\operatorname{dom}}\,}

\DeclareMathOperator*{\argmin}{argmin}
\newcommand{\ran}{\ensuremath{{\operatorname{ran}}\,}}
\newcommand{\zer}{\ensuremath{\operatorname{zer}}}

\newcommand{\Fix}{\ensuremath{\operatorname{Fix}}}
\newcommand{\Id}{\ensuremath{\operatorname{Id}}}

\newcommand{\pinf}{\ensuremath{+\infty}}

{\begin{list}{}{%
\settowidth{\labelwidth}{\textrm{#1~}}%
\setlength{\leftmargin}{\labelwidth+\labelsep}}}%requires macro calc.sty
{\end{list}}
\usepackage{amsthm}
\makeatletter% This is to fix the bold after the theorem
\def\th@plain{%
	\thm@notefont{}% same as heading font
	\itshape % body font
}
\def\th@definition{%
	\thm@notefont{}% same as heading font
	\normalfont % body font
}
\makeatother
\usepackage[capitalize,nameinlink]{cleveref}
\crefname{equation}{}{equations}
\crefname{chapter}{Appendix}{chapters}
\crefname{item}{}{items}
\crefname{enumi}{}{}
\newtheorem{theorem}{Theorem}[section]
\newtheorem{lemma}[theorem]{Lemma}

\newtheorem{corollary}[theorem]{Corollary}

\newtheorem{proposition}[theorem]{Proposition}

\newtheorem{definition}[theorem]{Definition}

\newtheorem{example}[theorem]{Example}
\newtheorem{ex}[theorem]{Example}
\newtheorem{fact}[theorem]{Fact}
\newtheorem{remark}[theorem]{Remark}

\providecommand{\norm}[1]{\lVert#1\rVert}

\providecommand{\RR}{\mathbb{R}}

\providecommand{\ran}{\operatorname{ran}}

\providecommand{\dom}{\operatorname{dom}}

\newcommand{\fix}{\ensuremath{\operatorname{Fix}}}

\providecommand{\gr}{\operatorname{gra}}

\providecommand{\Id}{\operatorname{{ Id}}}

\providecommand{\To}{\rightrightarrows}

\providecommand{\gr}{\operatorname{gra}}
\providecommand{\fix}{\operatorname{Fix}}
\providecommand{\ran}{\operatorname{ran}}

\providecommand{\Id}{\operatorname{Id}}

\providecommand{\zer}{\operatorname{zer}}
\providecommand{\R}{{ R}}
\newcommand{\cran}{\ensuremath{\overline{\operatorname{ran}}\,}}

\providecommand{\RR}{\mathbb{R}}

\definecolor{myblue}{rgb}{0.9,0.9,0.98}

  \newcommand*\mybluebox[1]{%
    \colorbox{myblue}{\hspace{1em}#1\hspace{1em}}}

\allowdisplaybreaks % or locally if problems {\allowdisplaybreaks
\newcommand*{\tran}{^{\mkern-1.5mu\mathsf{T}}}
\newcommand{\rddots}{\rotatebox[origin=lT]{10}{$\ddots$}}
\usepackage{relsize}

\DeclarePairedDelimiter{\parens}{\lparen}{\rparen}

\begin{document}

\setlength{\abovedisplayskip}{8pt}
\setlength{\belowdisplayskip}{8pt}	
\newsavebox\myboxA
\newsavebox\myboxB
\newlength\mylenA

\newcommand*\xoverline[2][0.75]{%
    \sbox{\myboxA}{$#2$}%
    \setbox\myboxB\null% Phantom box
    \ht\myboxB=\ht\myboxA%
    \dp\myboxB=\dp\myboxA%
    \wd\myboxB=#1\wd\myboxA% Scale phantom
    \sbox\myboxB{$\overline{\copy\myboxB}$}%  Overlined phantom
    \setlength\mylenA{\the\wd\myboxA}%   calc width diff
    \addtolength\mylenA{-\the\wd\myboxB}%
    \ifdim\wd\myboxB<\wd\myboxA%
       \rlap{\hskip 0.5\mylenA\usebox\myboxB}{\usebox\myboxA}%
    \else
        \hskip -0.5\mylenA\rlap{\usebox\myboxA}{\hskip 0.5\mylenA\usebox\myboxB}%
    \fi}
\makeatother

\makeatletter
\renewcommand*\env@matrix[1][\arraystretch]{%
  \edef\arraystretch{#1}%
  \hskip -\arraycolsep
  \let\@ifnextchar\new@ifnextchar
  \array{*\c@MaxMatrixCols c}}
\makeatother

\providecommand{\wbar}{\xoverline[0.9]{w}}
\providecommand{\ubar}{\xoverline{u}}

\newcommand{\nn}[1]{\ensuremath{\textstyle\mathsmaller{({#1})}}}
\newcommand{\crefpart}[2]{%
  \hyperref[#2]{\namecref{#1}~\labelcref*{#1}~\ref*{#2}}%
}
\newcommand\bigzero{\makebox(0,0){\text{\LARGE0}}}
	
%-------------------------------------------------------------------------

%\tikzstyle{decision} = [diamond, draw, fill=blue!50]
%\tikzstyle{line} = [draw, -stealth, thick]
%\tikzstyle{elli}=[draw, ellipse, fill=red!50,minimum height=8mm, text width=5em, text centered]
%\tikzstyle{block} = [draw, rectangle, fill=blue!50, text width=8em, text centered, minimum height=15mm, node distance=10em]
%

\author{
Heinz H.\ Bauschke\thanks{
Mathematics, University
of British Columbia,
Kelowna, B.C.\ V1V~1V7, Canada. E-mail:
\texttt{heinz.bauschke@ubc.ca}.},~
Walaa M.\ Moursi\thanks{
Department of Combinatorics and Optimization, 
University of Waterloo,
Waterloo, Ontario N2L~3G1, Canada.
%   and
%   Mansoura University, Faculty of Science,
%   Mathematics Department,
%   Mansoura 35516, Egypt.
 E-mail: \texttt{walaa.moursi@uwaterloo.ca}.}
,~
Shambhavi Singh\thanks{
Mathematics, University
of British Columbia,
Kelowna, B.C.\ V1V~1V7, Canada. E-mail:
\texttt{sambha@mail.ubc.ca}.},
~and~
Xianfu Wang\thanks{
Mathematics, 
University of British Columbia,
Kelowna, B.C.\ V1V~1V7, Canada.
 E-mail: \texttt{shawn.wang@ubc.ca}.
}
}

\title{\textsf{
On the Bredies-Chenchene-Lorenz-Naldi algorithm:\\ 
linear relations and strong convergence
}
}

\date{April 23, 2025}

\maketitle

\begin{abstract}
Monotone inclusion problems occur in many areas
of optimization and variational analysis. 
Splitting methods, which utilize resolvents or proximal mappings of the underlying operators, are often applied to solve these problems. 
In 2022, Bredies, Chenchene, Lorenz, and Naldi
introduced a new elegant algorithmic framework that 
encompasses various well known algorithms including 
Douglas-Rachford and Chambolle-Pock. They obtained 
powerful weak and strong convergence results, where the latter type relies on additional strong monotonicity assumptions. 

In this paper, we complement the analysis
by Bredies et al.\ by relating the projections of the fixed point sets of the underlying operators that generate
the (reduced and original) preconditioned proximal point
sequences. 
We obtain a new strong convergence result when the underlying operator is a linear relation. We note 
without assumptions such as linearity or strong monotonicity, 
one may encounter weak convergence without strong convergence.
In the case of the Chambolle-Pock algorithm, we obtain a new result that yields strong convergence 
to the projection onto the intersection of a linear subspace
and the preimage of a linear subspace. Splitting algorithms 
by Ryu and by Malitsky and Tam are also considered. 
Various examples are provided to illustrate the applicability of our results.

\end{abstract}
{ 
\small
\noindent
{\bfseries 2020 Mathematics Subject Classification:}
{Primary 
47H05, 
47H09;
Secondary 
47N10, 
65K05, 
90C25.
}

\noindent {\bfseries Keywords:}
Bredies-Chenchene-Lorenz-Naldi algorithm,
Chambolle-Pock algorithm,
Douglas-Rachford algorithm, 
Malitsky-Tam algorithm, 
maximally monotone operator, 
proximal point algorithm, 
Ryu's algorithm, 
splitting algorithm. 
}

\section{Introduction}

Throughout, we assume that 
\begin{empheq}[box=\mybluebox]{equation}
\text{$H$ is
a real Hilbert space with inner product 
$\scal{\cdot}{\cdot}\colon H\times H\to\RR$, }
\end{empheq}
and induced norm $\|\cdot\|$.
We are given a set-valued operator $A\colon H\To H$.
Our goal is to find a point $x$ in $\zer A$, 
which we assume to be nonempty: 
\begin{equation}
\label{e:thegoal}
\text{find $x\in H$ such that $0\in Ax$.}
\end{equation}
This is a very general and powerful problem in 
modern optimization and variational analysis; 
see, e.g., \cite{BC2017}. 
Following the framework proposed recently 
by Bredies, Chenchene, Lorenz, and Naldi \cite{BCLN} 
(see also the follow up \cite{BCN} and 
the related works\footnote{We would like to thank 
Dr.\ Panos Patrinos for bringing these to our attention.} \cite{LP}, \cite{EPLP}, and \cite{ELP}), 
we also assume that 
we are given another (possibly different) 
real Hilbert space $D$, as well as a
continuous linear operator 
\begin{empheq}[box=\mybluebox]{equation}
C\colon D\to H,
\end{empheq}
with adjoint operator $C^*\colon H\to D$. 
We shall (usually) assume that $C^*$ is surjective. 
Because $D = \ran C^* = \cran C^* = (\ker C)^\perp$,
it is clear that $\ker C= \{0\}$ and $C$ is injective. 
We thus  think of 
$D$ as a space that is ``smaller'' than $H$. 
Now set 
\begin{empheq}[box=\mybluebox]{equation}
\label{e:defM}
M := CC^*\colon H\to H, 
\end{empheq}
and one thinks of $M$ as a preconditioner\footnote{The viewpoint taken in \cite{BCLN} is to start with a preconditioner $M$ and then to construct the operator $C$; see \cite[Proposition~2.3]{BCLN} for details.}. 
It follows from \cite[Lemma~8.40]{Deutsch} that 
$\ran(M)=\ran(C)$ is closed. 
We shall assume that 
the set-valued map 
\begin{equation}
\label{e:Amono}
\text{$H\To H\colon 
x\mapsto Ax\cap\ran M$ is monotone,}
\end{equation}
which is clearly true if $A$ itself is monotone, 
and that 
\begin{equation}
\label{e:rfk}
\text{
$(A+M)^{-1}$ is single-valued, full-domain, and 
Lipschitz continuous. 
}
\end{equation}
This guarantees that the resolvent\footnote{We use the notation $\J{B}:=(\Id+B)^{-1}$ for the classical resolvent and $\R{B}:=2\J{B}-\Id$ for the reflected resolvent.} of $M^{-1}A$, 
\begin{empheq}[box=\mybluebox]{equation}
\label{e:defT}
T := (M+A)^{-1}M = (\Id+M^{-1}A)^{-1} = \J{M^{-1}A}\colon H\to H
\end{empheq}
is also single-valued, full-domain, and Lipschitz continuous\footnote{$M^{-1}A$ is not necessarily monotone and hence $T$ need not be firmly nonexpansive; however, it is ``$M$-monotone'' by \cref{e:Amono} and \cite[Proposition~2.5]{BCLN}. See also \cref{sss:2lines} below.}. 
The importance of $T$ for the problem
\cref{e:thegoal} is the relationship
\begin{equation}
\Fix T = \zer A.
\end{equation}
Associated with $T$, which operates in $H$,
is the operator (see \cite[Theorem~2.13]{BCLN} or \cite{BA-R})
\begin{empheq}[box=\mybluebox]{equation}
\label{e:defTt}
\Tt := C^*(M+A)^{-1}C 
= \J{C^*\tr A}\colon D\to D, 
\end{empheq}
which operates in $D$; here $C^*\tr A := (C^*A^{-1}C)^{-1}$ is a maximally monotone operator on $D$ and hence $\Tt$ is firmly nonexpansive. 
Furthermore, we always assume that 
\begin{empheq}[box=\mybluebox]{equation}
\text{$(\lambda_k)_\kkk$ is a sequence of real parameters in $[0,2]$ with $\sum_\kkk \lambda_k(2-\lambda_k)<\pinf$, }
\end{empheq}
and sometimes we will impose the more restrictive 
condition that 
\begin{equation}
\label{e:stayaway}
0<\inf_\kkk\lambda_k\leq \sup_\kkk\lambda_k<2. 
\end{equation}
In order to solve \cref{e:thegoal},
Bredies et al.\ \cite{BCLN} 
consider two sequences which are tightly intertwined:
The first sequence, generated by the so-called 
\emph{preconditioned proximal point (PPP) algorithm},  resides in $H$ and is given by 
\begin{empheq}[box=\mybluebox]{equation}
\label{e:genu}
u_{k+1} := 
(1-\lambda_k)u_k + \lambda_k Tu_k, 
\end{empheq}
where $u_0\in H$ is the starting point (see \cite[equation~(2.2)]{BCLN}).
The second sequence, generated by the so-called
\emph{reduced 
preconditioned proximal point (rPPP) algorithm}
resides in $D$ and is given by 
(see \cite[equation~(2.15)]{BCLN}) 
\begin{empheq}[box=\mybluebox]{equation}
\label{e:genw}
w_{k+1} := 
(1-\lambda_k)w_k + \lambda_k \Tt w_k, 
\end{empheq}
where $w_0 = C^*u_0$. 

With these sequences in place, we now state  
the main result on the PPP and rPPP algorithms from Bredies et al.

\begin{fact}[Bredies-Chenchene-Lorez-Naldi, 2022]
\label{f:known}
For every starting point $u_0\in H$, set $w_0 := C^*u_0$. 
Then there  exist $u^*\in\zer A = \Fix T$ and 
$w^*\in \Fix \Tt$ such that the following hold:
\begin{enumerate}
\item 
\label{f:known1}
The sequence $(Tu_k)_\kkk$ converges weakly to $u^*$.
\item 
\label{f:known2}
The sequence $(w_k)_\kkk$  converges weakly to $w^*$. 
\item 
\label{f:known3}
The sequence $((M+A)^{-1}Cw_k)_\kkk$ coincides with 
$(Tu_k)_\kkk$. 
\item 
\label{f:known4}
The sequence $(C^*u_k)_\kkk$ coincides with 
$(w_k)_\kkk$. 
\item 
\label{f:known5}
$u^*=(M+A)^{-1}Cw^*$. 
\item 
\label{f:known6}
$w^*=C^*u^*$. 
\item 
\label{f:known7}
If \cref{e:stayaway} holds, then $(u_k)_\kkk$ also converges weakly to $u^*$. 
\end{enumerate}
\end{fact}
\begin{proof}
\cref{f:known1}: 
See \cite[Corollary~2.10]{BCLN}. 
\cref{f:known2}: 
See \cite[Corollary~2.15]{BCLN}. 
\cref{f:known3}: 
This is buried in \cite[Proof of Corollary~2.15]{BCLN}. 
\cref{f:known4}: 
See \cite[Theorem~2.14]{BCLN}. 
\cref{f:known5}: 
Combine \cite[Corollary~2.15]{BCLN}, which states 
that $((M+A)^{-1}Cw_k)_\kkk$ converges weakly 
to $(M+A)^{-1}Cw^*$, with \cref{f:known3} and \cref{f:known1}.
\cref{f:known6}: 
This is buried in \cite[Proof of Corollary~2.15]{BCLN}. 
\cref{f:known7}: 
See \cite[Corollary~2.10]{BCLN}. 
\end{proof}

In \cite[Section~2.3]{BCLN} Bredies et al.\ also obtain
linear convergence under the usual assumptions on strong
monotonicity. Their framework encompasses an impressive number of other algorithms (see \cite[Section~3]{BCLN} for details) including Douglas-Rachford, Chambolle-Pock, and others. 

\emph{The goal of this paper is to complement the work by 
Bredies et al.\ \cite{BCLN}}. Specifically, our main results are:
\begin{enumerate}
\item[\bf R1] 
We introduce the notion of the $M$-projection onto $\Fix T$ (\cref{d:Mproj}) and we show how it is obtained from the (regular) projection onto $\Fix \Tt$ (\cref{t:u*}). 
\item[\bf R2]
We obtain \emph{strong} convergence results for PPP and rPPP sequences in the case when $A$ is a linear relation (\cref{t:sconv}). 
%\item[\bf R3]
\end{enumerate}
Assume $A$ is a linear relation. Then $T$ and $\Tt$ are linear operators and so trivially $0\in \Fix T$ and $0\in \Fix\Tt$. \textbf{R1} yields the exact limits, which 
are closest to the starting points (with respect to the seminorm induced by $M$) and which thus may be different from $0$, while \textbf{R2} guarantees strong convergence of the sequences generated by PPP and rPPP.

We also provide: 
(i) 
various applications to algorithms in the context of normal cone operators of closed linear subspaces and
(ii)
an example where the range of the preconditioner $M$ is not closed\footnote{The existence of such an example was proclaimed in \cite[Remark~3.1]{BCLN}.}. 

Finally, let us stress that in the absence of strong 
monotonicity and linearity, \emph{weak convergence may occur!} 
We shall demonstrate this for Douglas-Rachford and for Chambolle-Pock 
(see \cref{sss:DRnotstrong} and \cref{sss:CPnotstrong} below), by making use of the following classical example due 
to Genel and Lindenstrauss \cite{GL}:

\begin{fact}[Genel-Lindenstrauss]
\label{f:GL}
In $\ell^2$, there exists a nonempty bounded closed convex set $S$, 
a nonexpansive mapping $N \colon S\to S$, and a starting point 
$s_0 \in S$ such that the sequence generated by 
$s_{k+1} := \thalb s_k + \thalb Ns_k$ converges weakly
 --- but not strongly --- 
to a fixed point of $N$.
\end{fact}

The remainder of this paper is organized as follows.
In \cref{sec:aux}, we collect some auxiliary results which are required later. 
Our main result \textbf{R1} along with the limits of the sequences generated by the PPP and rPPP algorithms are discussed in \cref{sec:limits}. 
In \cref{sec:conv}, we present our main result \textbf{R2}. 
Various algorithms are analyzed as in \cref{sec:examples}, 
with an emphasis on the case when normal cone operators 
of closed linear subspaces are involved. 
Some additional results concerning the Chambolle-Pock algorithm are provided in \cref{s:moreCP}, including 
an example where the range of $M$ is not closed.

The notation we employ is fairly standard and follows 
largely \cite{BC2017}.

\section{Auxiliary results}
\label{sec:aux}

In this section, we shall collect various results that 
will make subsequent proofs more structured. 

\subsection{More results from Bredies et al.'s \cite{BCLN}}

The next result is essentially due to Bredies et al.; however, it is somewhat buried in
\cite[Proof of Theorem~2.14]{BCLN}: 

\begin{proposition}[$\Fix T$ vs $\Fix \Tt$]  
\label{p:FixTtilde}
Recall the definitions of $T$ and $\Tt$ from 
\cref{e:defT} and \cref{e:defTt}. Then: 
\begin{enumerate}
\item 
\label{p:FixTtilde1}
If $u\in\Fix T$, then $C^*u\in\Fix \Tt$. 
\item 
\label{p:FixTtilde2}
If $w\in\Fix \Tt$, then $w = C^*u$, where $u:=(M+A)^{-1}Cw\in\Fix T$. 
\item 
\label{p:FixTtilde3}
If $u\in\Fix T$, then $u= (M+A)^{-1}Cw$, 
where $w := C^*u\in\Fix\Tt$. 
\end{enumerate}
Consequently, 
the functions $C^*\colon \Fix T\to\Fix\Tt$ and 
$(M+A)^{-1}C \colon \Fix \Tt\to \Fix T$ are bijective and inverse of each other, and 
\begin{equation}
\label{e:FixTtilde}
C^*\big(\Fix T\big) = \Fix \Tt \;\;\text{and}\;\; 
(M+A)^{-1}C\big(\Fix \Tt\big) = \Fix T. 
\end{equation}
\end{proposition}
\begin{proof}
\cref{p:FixTtilde1}: 
Let $u\in\Fix T$ and set $w := C^*u$. 
Recalling \cref{e:defM}, we thus have the implications
$u=Tu=(M+A)^{-1}Mu= (M+A)^{-1}CC^*u$
$\Rightarrow$
$w = C^*u = (C^*(M+A)^{-1}C)(C^*u)= \Tt w$
$\Leftrightarrow$
$w\in\Fix \Tt$. 

\cref{p:FixTtilde2}: 
Let $w\in\Fix\Tt$. 
Clearly, $w\in\ran \Tt \subseteq \ran C^*$. 
Hence $Cw\in\ran CC^* = \ran M$ and so $(A+M)^{-1}Cw$ is single-valued. 
Next,
\begin{align*}
Tu 
&= T(M+A)^{-1}Cw \tag{by definition of $u$}\\
&= (M+A)^{-1}M(M+A)^{-1}Cw \tag{by \cref{e:defT}}\\
&= (M+A)^{-1}CC^*(M+A)^{-1}Cw \tag{by \cref{e:defM}}\\
&= (M+A)^{-1}C\Tt w \tag{by \cref{e:defTt}}\\
&= (M+A)^{-1}Cw \tag{because $w\in\Fix\Tt$}\\
&= u \tag{by definition of $u$}
\end{align*}
and 
so $u\in\Fix T$. 
Moreover, 
$
C^*u = C^*(M+A)^{-1}Cw=\Tt w = w$ as claimed.

\cref{p:FixTtilde3}: 
Let $u\in\Fix T$. 
By \cref{p:FixTtilde1}, $w\in\Fix\Tt$
and $(M+A)^{-1}Cw = (M+A)^{-1}CC^*u=(M+A)^{-1}Mu =
Tu = u$ as announced. 

Finally, \cref{p:FixTtilde1}--\cref{p:FixTtilde3} 
yield the ``Consequently'' part.
\end{proof}

\begin{fact}
\label{f:buried}
Consider the PPP algorithm.
If the parameter sequence $(\lambda_k)_\kkk$ satisfies \cref{e:stayaway}, 
then $u_k-Tu_k\to 0$.  
\end{fact}
\begin{proof}
See the second half of \cite[Proof of Theorem~2.9]{BCLN}. 
\end{proof}

Utilized in the work by Bredies et al.\ is the \emph{seminorm}
\begin{equation}
\label{e:semi}
\|x\|_M = \sqrt{\scal{x}{Mx}}= \|C^*x\|
\end{equation}
on $H$ 
which allows for some elegant formulations; 
see in particular \cite[Section~2]{BCLN}.

\subsection{A strong convergence result}

\begin{fact}[Baillon-Bruck-Reich, 1978]
\label{f:BBR}
Let $N\colon D\to D$ be a linear nonexpansive mapping such that 
$N^kx_0-N^{k+1}x_0\to 0$ for every $x_0\in D$. 
Then $(N^kx_0)_\kkk$ converges \emph{strongly} to a
fixed point of $N$. 
\end{fact}
\begin{proof}
See the original \cite[Theorem~1.1]{BBR}, 
or \cite[Theorem~5.14(ii)]{BC2017}. 
\end{proof}

\subsection{From Genel-Lindenstrauss to maximally monotone operators}

It will be convenient to extend and repackage \cref{f:GL} as follows:

\begin{proposition}
\label{p:GL}
There exists a maximally monotone operator 
$B$ on $\ell^2$, with bounded domain, and a starting point $s_0\in \ell^2$ 
such that 
the sequence $(J_{B}^ks_0)_\kkk$ converges weakly 
--- but not strongly --- to some point $\bar{s}\in \zer B$. 
\end{proposition}
\begin{proof}
Let $S$, $N$, and $s_0$ be as in \cref{f:GL}. 
(We recall that $\Fix N\neq\varnothing$ by the fixed point theorem of 
Browder-G\"ohde-Kirk; see, e.g., \cite[Theorem~4.29]{BC2017}.)
The operator $F := \thalb\Id+\thalb N\colon S\to S$ is firmly nonexpansive, and $(F^ks_0)_\kkk$ converges weakly 
--- but not strongly --- to some point in $\Fix N = \Fix F$. 
By \cite[Corollary~5]{HBext}, there exists 
an extension $\widetilde{F}\colon \ell^2\to S$ such that 
$\widetilde{F}$ is firmly nonexpansive and 
$\widetilde{F}|_S = F$. 
On the one hand, $\Fix F \subseteq \fix\widetilde{F}$.
On the other hand, if $x\in \Fix\widetilde{F}$,
 then $x=\widetilde{F}x \in\ran\widetilde{F}\subseteq S$ 
and so $x=\widetilde{F}|_Sx = Fx$, i.e., $x\in\Fix F$.
Altogether, 
$\varnothing\neq \Fix F = \Fix \widetilde{F} \subseteq S$.
Finally, let $B := \widetilde{F}^{-1}-\Id$. 
Then $\dom B = \ran \widetilde{F} \subseteq S$ is bounded, 
$B$ is maximally monotone on $\ell^2$, 
$\varnothing \neq \zer B = \Fix \widetilde{F}$, 
and $J_{B} = \widetilde{F}$. 
The result thus follows from \cref{f:GL}. 
\end{proof}

\section{The limits and the $M$-projection}
\label{sec:limits}

\subsection{The limits}

\label{sec:limitslimits}

In this section, we identify the limits $w^*$ and $u^*$ of the rPPP and PPP algorithms (see \cref{f:known}) provided that additional assumptions are satisfied. 
Our result rests on the following relationship
between $P_{\Fix \Tt}$ and the
projection onto $\Fix T$ with respect to the 
seminorm from \cref{e:semi}.

\begin{theorem}[projecting onto $\Fix\Tt$ and $\Fix T$]
\label{t:krista}
Let $u_0\in H$ and assume that $w_0 = C^*u_0$. 
Then the following hold:
\begin{enumerate}
\item
\label{t:krista1}
If $\widehat{w} := P_{\Fix\Tt}(w_0)$, 
then $\widehat{u} := (M+A)^{-1}C\widehat{w}$ is an $M$-projection of $u_0$ onto 
$\Fix T$ in the sense that 
\begin{equation}
(\forall x\in\Fix T)\quad 
\|u_0-\widehat{u}\|_M \leq \|u_0-x\|_M.
\end{equation}
\item
\label{t:krista2}
If $\widehat{u}\in \Fix T$ 
is an $M$-projection of $u_0$ onto 
$\Fix T$ in the sense that 
$(\forall x\in \Fix T)$
$\|u_0-\widehat{u}\|_M \leq \|u_0-x\|_M$, 
then 
$\widehat{w} := C^*\widehat{u}$ is equal to 
$P_{\Fix\Tt}(w_0)$. 
\item
\label{t:krista3}
If $\widehat{u_1},\widehat{u_2}$ belong to $\Fix T$ and are both $M$-projections 
of $u_0$ onto $\Fix T$, then $\widehat{u_1}=\widehat{u_2}$. 
\end{enumerate}
\end{theorem}
\begin{proof}
\cref{t:krista1}:
Let $x\in\Fix T$ and set 
$y:=C^*x$.
Then $y\in\Fix\Tt$ by \cref{p:FixTtilde}\cref{p:FixTtilde1}. 
Hence $\|w_0-\widehat{w}\|\leq\|w_0-y\|$. 
On the other hand,
$w_0=C^*u_0$ (by assumption),  
$\widehat{w}=C^*\widehat{u}$ 
and  $\widehat{u}\in \Fix T$
(by \cref{p:FixTtilde}\cref{p:FixTtilde2}).  
Altogether, we obtain 
$\|C^*u_0-C^*\widehat{u}\|\leq\|C^*u_0-C^*x\|$ or 
$\|C^*(u_0-\widehat{u})\|\leq\|C^*(u_0-x)\|$. 
The conclusion now follows from 
\cref{e:semi}.

\cref{t:krista2}:
By \cref{p:FixTtilde}\cref{p:FixTtilde1},
we have $\widehat{w}\in\Fix\Tt$. 
Let $y\in\Fix\Tt$ and set
$x := (M+A)^{-1}Cy$.
By \cref{p:FixTtilde}\cref{p:FixTtilde2},
we have $x\in\Fix T$ and $y=C^*x$. 
Moreover, 
$C^*(u_0-x)=C^*u_0-C^*x=w_0-y$ and 
$C^*(u_0-\widehat{u})= C^*u_0-C^*\widehat{u}=w_0-\widehat{w}$.
The assumption that 
$\|u_0-\widehat{u}\|_M\leq\|u_0-x\|_M$ turns, 
in view of \cref{e:semi} and the above, into 
$\|w_0-\widehat{w}\|\leq \|w_0-y\|$. 
Hence $\widehat{w} = P_{\Fix \Tt}(w_0)$.

\cref{t:krista3}:
By \cref{t:krista2}, 
$C^*\widehat{u_1} = P_{\Fix\Tt}(w_0)=
C^*\widehat{u_2}$.
Hence $\widehat{u_2} = \widehat{u_1}+k$,
where $k\in\ker C^*$. 
It follows that
\begin{align}
\widehat{u_2} &= T\widehat{u_2} \tag{because $\widehat{u_2}\in\Fix T$}\\
&=(A+M)^{-1}(CC^*)(\widehat{u_1}+k)
\tag{by \cref{e:defT} and \cref{e:defM}}\\
&=(A+M)^{-1}(CC^*\widehat{u_1}+CC^*k)\tag{$CC^*$ is linear}\\
&=(A+M)^{-1}(M\widehat{u_1}+0) \tag{by \cref{e:defM} and because $k\in\ker C^*$}\\
&=(A+M)^{-1}M\widehat{u_1}\notag\\
&=T\widehat{u_1}\tag{by \cref{e:defT}}\\
&=\widehat{u_1}\tag{because $\widehat{u_1}\in\Fix T$}
\end{align}
and we are done. 
\end{proof}

In view of \cref{t:krista}, the following notion is now well defined:

\begin{definition}[$M$-projection onto $\Fix T$]
\label{d:Mproj}
For every $u_0\in H$, there exists a unique point 
$\widehat{u}\in\Fix T$ such that 
$(\forall x\in \Fix T)$
$\|u_0-\widehat{u}\|_M \leq \|u_0-x\|_M$. 
The point $\widehat{u}$ is called the 
\emph{$M$-projection of $u_0$ onto $\Fix T$}, 
written also as $P_{\Fix T}^M(u_0)$. 
\end{definition}

\begin{remark}
It is the special structure of $\Fix T$ that allows us 
to introduce the single-valued and full-domain operator 
$P_{\Fix T}^M$. For more general sets, 
$M$-projections either may not exist or they may not be a singleton
--- see \cref{sss:bizarre} below.
\end{remark}

We are now ready for a nice sufficient condition that 
allows for the identification of the weak limit $w^*$ of the sequence generated by the rPPP algorithm: 

\begin{theorem}[when $\Fix\Tt$ is affine]
\label{t:w*}
Suppose that $\Fix \Tt$ is an affine subspace of $D$. 
Then 
\begin{equation}
\label{e:w*}
w^*= P_{\Fix\Tt}(w_0). 
\end{equation}
\end{theorem}
\begin{proof}
By \cite[Corollary~5.17(i)]{BC2017},
the sequence $(w_k)_\kkk$ is Fej\'er monotone with respect to $\Fix\Tt$.
From \cref{f:known}\cref{f:known2}, 
we know that $(w_k)_\kkk$ converges weakly to 
$w^*\in\Fix\Tt$.
Therefore, \cite[Proposition~5.9(ii)]{BC2017} yields
that $(w_k)_\kkk$ converges weakly to 
$P_{\Fix\Tt}(w_0)$.
Altogether, we deduce \cref{e:w*}. 
\end{proof}

\begin{remark}[linear relations]
\label{r:linrels}
Suppose that our given operator $A$ is actually 
a \emph{linear relation}, i.e., its graph 
is a linear subspace of $H\times H$. 
(For more on linear relations, we recommend 
Cross's monograph \cite{Cross} 
and Yao's doctoral thesis \cite{Yao}.)
Then $\Tt$ (see \cref{e:defTt}) is actually a
firmly nonexpansive linear operator; consequently,
its fixed point set $\Fix \Tt$ is a linear subspace 
and \cref{t:w*} is applicable.
Similar comments hold for the case when 
$A$ is an affine relation, i.e., its graph is an affine subspace of $H\times H$. 
\end{remark}

When the weak limit $w^*$ of the rPPP sequence $(w_k)_\kkk$ 
is $P_{\Fix \Tt}(w_0)$, 
then we are able to identify 
the weak limit $u^*$ of
the PPP sequence $(Tu_k)_\kkk$ as $P_{\Fix T}^M(u_0)$: 

\begin{theorem}[$P_{\Fix T}^M$ from $P_{\Fix\Tt}$]
\label{t:u*}
Suppose that $w^* = P_{\Fix\Tt}(w_0)$. 
Then 
\begin{equation}
\label{e:u*}
u^* = P_{\Fix T}^M(u_0). 
\end{equation}
\end{theorem}
\begin{proof}
On the one hand, 
$(M+A)^{-1}Cw^*= (M+A)^{-1}CP_{\Fix\Tt}(w_0) = P_{\Fix T}^M(u_0)$
by \cref{t:krista}\cref{t:krista1}. 
On the other hand, 
$u^*=(M+A)^{-1}Cw^*$ by \cref{f:known}\cref{f:known5}. 
Altogether, we obtain the announced identity~\cref{e:u*}. 
\end{proof}

\subsection{On the notion of a general $M$-projection}
\label{s:limitsMproj}

\begin{proposition}
\label{p:jen}
Let $S$ be a nonempty subset of $H$, and let $h\in H$.
Then\footnote{Here we set $\Pi_S^M(h)=\argmin_{s\in S}\|h-s\|_M$ and $\Pi_R(d)=\argmin_{r\in R}\|d-r\|$. }
\begin{equation}
\Pi_S^M(h) = S \cap (C^*)^{-1}\big(\Pi_{C^*(S)}(C^*h)\big).
\end{equation}
\end{proposition}
\begin{proof}
Suppose that $\hat{s}\in H$. 
Set $d := C^*h$, $\hat{r} := C^*\hat{s}$,
and $R := C^*(S)$. 
We have the following equivalences:
\begin{subequations}
\begin{align} 
\hat{s}\in \Pi_S^M(h)
&\Leftrightarrow
\hat{s}\in S\;\land\; 
(\forall s\in S)\; \|h-\hat{s}\|_M\leq\|h-s\|_M\\
&\Leftrightarrow
\hat{s}\in S\;\land\; 
(\forall s\in S)\; \|C^*h-C^*\hat{s}\|\leq\|C^*h-C^*s\|\\
&\Leftrightarrow
\hat{s}\in S\;\land\; \hat{r}\in R\;\land\; 
(\forall r\in R)\; \|d-\hat{r}\|\leq\|d-r\|\\
&\Leftrightarrow
\hat{s}\in S\;\land\;
\hat{r}\in\Pi_R(d)\\
&\Leftrightarrow
\hat{s}\in S\;\land\; C^*\hat{s}\in \Pi_{C^*(S)}(C^*h)\\
&\Leftrightarrow
\hat{s}\in S\;\land\; \hat{s}\in (C^*)^{-1}\Pi_{C^*(S)}(C^*h)\\
&\Leftrightarrow
\hat{s}\in S \cap (C^*)^{-1}\big(\Pi_{C^*(S)}(C^*h)\big) ,
\end{align}
\end{subequations}
and we are done!
\end{proof}

\begin{proposition}
\label{p:2donkeys}
The following hold:
\begin{enumerate}
\item
\label{p:2donkeys1}
If
$\Pi_S^M$ is at most singleton-valued, 
then 
\begin{equation}
\label{e:2donkeys}
(\forall s\in S)\quad
(s+\ker C^*)\cap S = \{s\}. 
\end{equation}
\item
\label{p:2donkeys2}
If $\Pi_{C^*(S)}$ is at most singleton-valued and 
\cref{e:2donkeys} holds, then 
$\Pi_S^M$ is at most singleton-valued. 
\end{enumerate}
\end{proposition}
\begin{proof}
Clearly, $(s+\ker C^*)\cap S\supseteq \{s\}$.

\cref{p:2donkeys1}: 
We prove the contrapositive. 
Suppose \cref{e:2donkeys} does not hold.
Then there exists $s\in S$ such that 
$(s+\ker C^*)\supsetneqq\{s\}$. 
Take $k\in\ker C^*\smallsetminus\{0\}$ such that $s+k = s'\in S$.
Then $\|s-s\|_M=\|0\|_M=\|C^*0\|=0$ and 
$\|s-s'\|_M = \|-k\|_M=\|-C^*k\|=0$.
Hence $\{s,s'\}\subseteq \Pi_S^M(s)$ and $\Pi_S^M$ is not at most singleton-valued. 

\cref{p:2donkeys2}: 
Let $h\in H$ and suppose that 
$s_1$ and $s_2$ both belong to $\Pi_S^M(h)$.
By \cref{p:jen},
$C^*s_1\in\Pi_{C^*(S)}(C^*h)$ and 
$C^*s_2\in\Pi_{C^*(S)}(C^*h)$. 
Because $\Pi_{C^*(S)}$ is at most singleton-valued, we deduce that 
$C^*s_1=C^*s_2$.
Hence $s_2-s_1\in\ker C^*$. 
Combining with \cref{e:2donkeys}, we deduce that 
$s_2\in(s_1+\ker C^*)\cap S = \{s_1\}$ and so
$s_2=s_1$. 
\end{proof}

\begin{remark}
The general results above give another view on 
why $P_{\Fix T}^M$ is not multi-valued. We sketch here why.
Suppose that $S = \Fix T$. 
We know that $C^*(S)=\Fix\Tt$ is a nonempty closed convex set because $\Tt$ is (firmly) nonexpansive.
Hence $\Pi_{C^*(S)}$ is actually singleton-valued.
Let $s\in S =\Fix T$, $k\in\ker C^*$, and assume that 
$s+k\in S$.
Then 
$s+k=T(s+k)=(A+M)^{-1}(Ms+Mk)=(A+M)^{-1}(Ms)=T(s)=s$
and so $k=0$.
This verifies \cref{e:2donkeys}.
In view of \cref{p:2donkeys}\cref{p:2donkeys2}, 
we see that $\Pi^M_S$ is (at most) singleton-valued.
\end{remark}

For an example that illustrates
that $\Pi_S^M$ may be empty-valued or multi-valued, 
see \cref{sss:bizarre} below.

\section{Convergence}

\label{sec:conv}

\begin{theorem}[weak convergence and limits]
\label{t:wconv}
Suppose that $\Fix \Tt$ is an affine subspace of $D$.
Let $u_0\in H$ and set $w_0 = C^*u_0$. Then the following hold for the rPPP and PPP sequences 
$(w_k)_\kkk$ and $(u_k)_\kkk$ generated by 
\cref{e:genw} and \cref{e:genu}:
\begin{enumerate}
\item 
\label{t:wconv1}
The sequence $(w_k)_\kkk$ converges weakly to 
$w^*=P_{\Fix\Tt}(w_0)$.\\[-3mm] 
\item 
\label{t:wconv2}
The sequence $(Tu_k)_\kkk$ converges weakly to 
$u^*=P^M_{\Fix T}(u_0) = (M+A)^{-1}Cw^*$.\\[-3mm]
\item 
\label{t:wconv3}
If \cref{e:stayaway} holds, then 
the sequence $(u_k)_\kkk$ also converges weakly to 
$u^*=P^M_{\Fix T}(u_0)= (M+A)^{-1}Cw^*$. 
\end{enumerate}
\end{theorem}
\begin{proof}
\cref{t:wconv1}: 
Combine 
\cref{f:known}\cref{f:known2} and \cref{t:w*}.
\cref{t:wconv2}: 
Combine
\cref{f:known}\cref{f:known1}, \cref{t:w*}, 
\cref{t:u*}, and \cref{t:krista}\cref{t:krista1}. 
\cref{t:wconv3}: 
Combine \cref{t:wconv2} with \cref{f:known}\cref{f:known7}.
\end{proof}

\begin{theorem}[strong convergence and limits]
\label{t:sconv}
Suppose that $A$ is a linear relation. 
Let $u_0\in H$ and set $w_0 = C^*u_0$. 
Suppose that the parameter sequence 
$(\lambda_k)_\kkk$ is identical to some constant $\lambda\in\left]0,2\right[$. 
Then the following hold for the rPPP and PPP sequences 
$(w_k)_\kkk$ and $(u_k)_\kkk$ generated by 
\cref{e:genw} and \cref{e:genu}:
\begin{enumerate}
\item 
\label{t:sconv1}
The sequence $(w_k)_\kkk$ converges strongly to 
$w^*=P_{\Fix\Tt}(w_0)$.\\[-3mm] 
\item 
\label{t:sconv2}
The sequences $(Tu_k)_\kkk$ and $(u_k)_\kkk$ converge strongly to 
$u^*=P^M_{\Fix T}(u_0)= (M+A)^{-1}Cw^*$. % \\[-3mm]
\end{enumerate}
\end{theorem}
\begin{proof}
Because $A$ is a linear relation on $H$, 
it follows that 
$\Tt$ is a linear operator and that $\Fix\Tt$ is a linear subspace of $D$. 
Note that $(\forall\kkk)$ $w_{k+1} = Nw_k$, 
where $N := (1-\lambda)\Id+\lambda\Tt 
= (1-(\lambda/2))\Id+(\lambda/2)\widetilde{N}$
and $\widetilde{N}:=2N-\Id$ are both linear and nonexpansive,
 with $\Fix\Tt = \Fix N = \Fix\widetilde{N}$. 
By \cite[Theorem~5.15(ii)]{BC2017},
$w_k-\widetilde{N}w_k\to 0$;
equivalently, $w_k-Nw_k\to 0$. 

\cref{t:sconv1}: 
Indeed, it follows from \cref{f:BBR} 
that $(w_k)_\kkk$ converges strongly. 
The result now follows from \cref{t:wconv}\cref{t:wconv1} 
(or \cite[Proposition~5.28]{BC2017}). 

\cref{t:sconv2}: 
On the one hand, using \cref{e:rfk}, 
we see that $(M+A)^{-1}C$ is Lipschitz continuous.
On the other hand, 
$(Tu_k)_\kkk = ((M+A)^{-1}Cw_k)_\kkk$ by 
\cref{f:known}\cref{f:known3}. 
Altogether, because $(w_k)_\kkk$ strongly converges
by \cref{t:sconv1}, we deduce that 
$(Tu_k)_\kkk$ must converge \emph{strongly} as well.
By \cref{t:wconv}\cref{t:wconv2}, $Tu_k\to P_{\Fix T}^M(u_0)$. 
Finally, piggybacking on \cref{f:buried}, we obtain that 
$(u_k)_\kkk$ converges strongly to $P_{\Fix T}^M(u_0)$ 
as well. 
\end{proof}
\begin{remark}
  Using a translation argument similar to what was done in \cite[Section~4.4]{BSW}, one can generalize \cref{t:sconv} to the case when $A$ is an affine relation, i.e., $\gr A$ is an affine subspace of $H\times H$, and such that $\zer A\neq 
\varnothing$. 
 \end{remark}

\section{Examples}

\label{sec:examples}

In this section, we consider the
Douglas-Rachford, Chambolle-Pock, Ryu, and the
Malitsky-Tam minimal lifting (which we will refer to as
Malitsky-Tam)
splitting algorithms. 

\subsection{Douglas-Rachford}

\subsubsection{General setup}

\label{sss:DRgeneral}

Following \cite[Sections~1 and 3]{BCLN}, we suppose that $X$ is a real Hilbert space,
and $A_1,A_2$ are two maximally monotone operators on $X$.
The goal is to find a point in $\zer(A_1+A_2)$. 
Now assume\footnote{We shall employ frequently ``block operator'' notation for convenience.} that 
\begin{equation}
H=X^2, \;\; D=X, \;\;\text{and} \;\;C=\begin{bmatrix}
\Id\\-\Id
\end{bmatrix}. 
\end{equation}
Then 
\begin{equation}
\label{e:C*DRS}
C^*=\begin{bmatrix}
\Id \;\; -\Id
\end{bmatrix}\colon X^2\to X \colon \begin{bmatrix}x\\y\end{bmatrix}
\mapsto x-y
\end{equation}
is clearly surjective. 
Now we compute 
\begin{equation}\label{e:MDRS}
M=CC^\ast=\begin{bmatrix}
\Id&-\Id\\-\Id& \Id
\end{bmatrix}
\end{equation}
and we set 
\begin{equation}\label{e:ADRS}
A=\begin{bmatrix}
 A_1 & \Id\\ -\Id & A_2^{-1}
\end{bmatrix}
\end{equation}
which 
is maximally monotone as the sum of a maximally monotone operator $(x,y)\mapsto A_1x\times A_2^{-1}y$ and a skew linear operator. 
Next\footnote{We use occasionally use the notation $B^\ovee=(-\Id)\circ B\circ (-\Id)$ and $B^{-\ovee}={(B^{-1})}^\ovee$.}, 
\begin{subequations}
\label{e:zerADRS}
\begin{align}
\zer A &= 
\menge{(x,y)}{0\in A_1x+y \land 0\in -x+A_2^{-1}y}\\ 
&= 
\menge{(x,y)}{0\in A_1x+y \land y\in A_2x}\\ 
&= 
\menge{(x,y)}{x\in\zer(A_1+A_2) \land -y\in A_1x\cap(-A_2x)}\\
&= 
\menge{(x,y)}{-y \in\zer(A_1^{-1}+A_2^{-\ovee}) \land x \in A_1^{-1}(-y)\cap (-A_2^{-\ovee}(-y))}\label{e:FixDRdual}\\
&=\gr(A_2)\cap \gr(-A_1)
\end{align}
\end{subequations}
where \cref{e:FixDRdual} follows from \cite[Proposition~2.4]{BBHM} 
and also relates to 
Attouch-Th\'era duality 
and associated operations (see \cite[Section~3]{BBHM}
for details). 
One verifies that 
\begin{equation}\label{e:MAinvDRS}
(M+A)^{-1}\colon X^2\to X^2\colon 
\begin{bmatrix} x\\ y\end{bmatrix} \mapsto \begin{bmatrix}
\J{ A_1}(x)\\\J{ A_2^{-1}}(y+2\J{A_1} (x))
\end{bmatrix}
\end{equation}
is indeed single-valued, full-domain, and Lipschitz continuous; hence,
\begin{equation}
\label{e:MAinvCDRS}
(M+A)^{-1}C\colon X\to X^2\colon 
w \mapsto 
\begin{bmatrix}
\J{ A_1}(w)\\\J{ A_2^{-1}}(-w+2\J{A_1}(w))
\end{bmatrix}
=
\begin{bmatrix}
\J{ A_1}(w)\\
\J{A_2^{-1}}\R{A_1}(w) 
\end{bmatrix}.
\end{equation}
We now compute
\begin{subequations}
\label{e:TDRS}
\begin{align}
T(x,y) &= (M+A)^{-1}M(x,y) 
= (M+A)^{-1}(x-y,y-x)\\
&= \begin{bmatrix}
\J{A_1}(x-y)\\
\J{A_2^{-1}}\R{A_1}(x-y)
\end{bmatrix}.
\end{align}
\end{subequations}
Furthermore, recalling \cref{e:C*DRS} and 
\cref{e:MAinvCDRS}, we obtain 
\begin{subequations}
\label{e:TtDRS}
\begin{align}
\Tt(w)
&= C^*(M+A)^{-1}Cw = 
[\Id\;\;-\Id]\begin{bmatrix}
\J{ A_1}(w)\\-w+2\J{A_1}(w)-\J{A_2}\R{A_1}(w) 
\end{bmatrix}\\
&= w - \J{A_1}(w) + \J{A_2}\R{A_1}(w), 
\end{align}
\end{subequations}
which is the familiar Douglas-Rachford splitting operator. 
Using the fact that $\Fix \Tt = C^*(\Fix T)$ (see \cref{e:FixTtilde}) and \cref{e:zerADRS}, we
obtain 
\begin{equation}
\label{e:FixTtDRS}
\Fix\Tt = 
\bigcup_{x\in\zer(A_1+A_2)}x+\big(A_1x\cap(-A_2x)\big), 
\end{equation}
an identity that also follows from 
\cite[Theorem~4.5]{BBHM}.

\subsubsection{An example without strong convergence}
\label{sss:DRnotstrong}
Now suppose temporarily that 
$X=\ell^2$, that $B$, $s_0$ and $\bar{s}$ are as in \cref{p:GL},
that $A_1 = B$, and 
that $A_2=0$. 
Then \cref{e:zerADRS}, \cref{e:TDRS}, and \cref{e:TtDRS} turn into 
\begin{equation}
\zer A = \zer(B)\times\{0\},
\;\;
T(x,y) = \begin{bmatrix}
\J{B}(x-y)\\
0
\end{bmatrix}, 
\;\;\text{and}\;\;
\Tt(w) = J_{B}(w). 
\end{equation}
Suppose that furthermore $(\forall\kkk)$ $\lambda_k=1$ and 
$u_0 = (s_0,0)$.  
It then follows that $w_0=C^*u_0 = s_0$,  
$u_k = T^k(s_0,0) = (J_{B}^k(s_0),0)\weakly (\bar{s},0)$ and 
$w_k = \Tt^k(s_0) = J^k_{B}(s_0)\weakly \bar{s}$ and 
neither convergence is strong.

\subsubsection{Specialization to normal cones of linear subspaces}

We now assume that $U_1,U_2$ are two closed linear 
subspaces of $X$ and that\footnote{Given a nonempty closed convex subset $U$ of $X$, recall that $N_U(x):=\begin{cases}
\menge{x^*\in X}{\max\scal{U-x}{x^*}= 0},&\text{if }x\in U;\\
\emptyset,&\text{otherwise} 
\end{cases}
$
is the normal cone operator of $U$
at $x\in X$.} $A_1=N_{U_1},A_2=N_{U_2}$.
Then $\zer(A_1+A_2)=U_1\cap U_2$, 
$\gr A_1 = U_1\times U_1^\perp = U_1\times (-U_1^\perp)=\gr (-A_1)$ and 
$\gr A_2 = U_2\times U_2^\perp$; hence,
\cref{e:zerADRS} yields
\begin{equation}
\Fix T = \zer A = (U_1\times U_1^\perp) \cap (U_2\times U_2^\perp) = (U_1\cap U_2)\times (U_1^\perp \cap U_2^\perp). 
\end{equation}
Next, we see that \cref{e:TDRS}, \cref{e:TtDRS}, \cref{e:FixTtDRS} turn into
\begin{align}
T(x,y) = \begin{bmatrix}
P_{U_1}(x-y)\\
P_{U_2^\perp}(P_{U_1}-P_{U_1^\perp})(x-y)
\end{bmatrix},
\end{align}
\begin{align}
\Tt = \Id - P_{U_1}+P_{U_2}(P_{U_1}-P_{U_1^\perp}) = P_{U_2}P_{U_1}+P_{U_2^\perp}P_{U_1^\perp},
\end{align}
\begin{align}
\Fix \Tt = 
\bigcup_{x\in U_1\cap U_2}
x+(U_1^\perp\cap (-U_2^\perp))
= (U_1\cap U_2) + (U_1^\perp\cap U_2^\perp)
\end{align}
respectively.
(The latter two identities were also derived in \cite{BBCNPW}.)
It follows from \cref{t:wconv}\cref{t:wconv1} that
the rPPP sequence, i.e., 
the governing sequence of the Douglas-Rachford algorithm, satisfies
\begin{equation}
\label{e:w*DRS}
w_k\weakly w^* = P_{\Fix\Tt}(w_0) = P_{U_1\cap U_2}(w_0)
+ P_{U_1^\perp\cap U_2^\perp}(w_0). 
\end{equation}
Combining \cref{f:known}\cref{f:known3}, 
\cref{e:MAinvCDRS}, \cref{t:wconv}\cref{t:wconv2},
\cref{t:krista}\cref{t:krista1},
\cref{e:MAinvCDRS}, and
\cref{e:w*DRS}, we have\footnote{Given a nonempty closed convex subset $U$ of $X$, its reflected projection is $\R{U}:=\R{N_U}:=2 P_U-\Id$. }
\begin{align}
Tu_k &= (M+A)^{-1}Cw_k
= \begin{bmatrix}
P_{U_1}w_k\\
P_{U_2^\perp}\R{U_1}w_k\\
\end{bmatrix}
\end{align}
and 
\begin{subequations}
\label{e:230628d}
\begin{align}
u^* &= P^M_{\Fix T}(u_0)
= (M+A)^{-1}Cw^*
= \begin{bmatrix}
P_{U_1}w^*\\
P_{U_2^\perp}\R{U_1}w^*
\end{bmatrix}
= 
\begin{bmatrix}
P_{U_1}\big(P_{U_1\cap U_2}(w_0) + P_{U_1^\perp\cap U_2^\perp}(w_0)\big)\\
P_{U_2^\perp}\R{U_1}\big(P_{U_1\cap U_2}(w_0) + P_{U_1^\perp\cap U_2^\perp}(w_0)\big)
\end{bmatrix}\\
&= 
\begin{bmatrix}
P_{U_1\cap U_2}(w_0)\\
-P_{U_1^\perp\cap U_2^\perp}(w_0)
\end{bmatrix}
= 
\begin{bmatrix}
P_{U_1\cap U_2}(x_0-y_0)\\
P_{U_1^\perp\cap U_2^\perp}(y_0-x_0)
\end{bmatrix}.
\end{align}
\end{subequations}
Note that the last description of $u^*$ clearly differs from $P_{\Fix T}(u_0)=
(P_{U_1\cap U_2}(x_0),P_{U_1^\perp\cap U_2^\perp}(y_0))$ in general!
Returning to \cref{e:230628d}, we deduce in particular that 
the shadow sequence 
$(P_{U_1}w_k)$ satisfies
\begin{equation}
P_{U_1}w_k\weakly P_{U_1\cap U_2}(w_0).
\end{equation}
Moreover, combining with \cref{t:sconv},
we have the following \emph{strong} convergence\footnote{
This was also derived in \cite{BBCNPW} when $\lambda=1$.}
result: 
\begin{equation}
w_k\to P_{U_1\cap U_2}(w_0)+P_{U_1^\perp\cap U_2^\perp}(w_0)
\;\;\text{and}\;\;
P_{U_1}w_k\to P_{U_1\cap U_2}(w_0)
\end{equation}
provided that $\lambda_k=\lambda$ for all $\kkk$.

\subsubsection{When $\Pi_S^M$ is bizarre}

\label{sss:bizarre}

We return to the general setup of 
\cref{sss:DRgeneral}.
We shall show that $\Pi_S^M$ may be empty or single-valued when $S$ is a certain closed convex subset of $H=X^2$. 
To this end, assume that $S \subseteq H$ is the Cartesian product 
\begin{equation}
S=S_1\times S_2,
\end{equation}
where $S_1,S_2$ are nonempty closed convex subsets of $X$.
Then $C^*(S)=S_1-S_2$ may fail to be closed
\emph{even when each $S_i$ is a closed linear subspace of $X$} (see, e.g., \cite{BBCNPW}). 
Consider such a scenario, 
let $\bar{x}\in\overline{S_1-S_2}\smallsetminus(S_1-S_2)$, and set $\bar{h} := (\bar{x},0)$ and let 
$s=(s_1,s_2)\in S$. 
Then $\|\bar{h}-s\|_M=\|\bar{x}-(s_1-s_2)\|>0$ while clearly $\inf\|\bar{x}-(S_1-S_2)\|=0$. In other words,
\begin{equation}
\Pi_S^M(\bar{h})=\varnothing. 
\end{equation}
On the other hand, if 
$S_1=S_2=X$ and $h=(h_1,h_2)\in S=H$, then
$S=S_1\times S_2 = X\times X$ and 
\cref{p:jen} yields
\begin{equation}
\Pi_S^M({h})=\big(h+\ker C^*\big) \cap S 
=
\big((h_1,h_2)+\menge{(x,x)}{x\in X}\big)
\end{equation}
is clearly multi-valued provided $X\neq\{0\}$.
In summary, 
$\Pi_S^M$ may be empty-valued or multi-valued.

\subsection{Chambolle-Pock}

\label{ss:CP}

\subsubsection{General setup}

\label{sss:CPgeneral}

Following the presentation of the algorithm by Chambolle-Pock (see \cite{CP}) in Bredies et al.'s \cite{BCLN} (see also \cite{Condat}), we suppose that $X$ and $Y$ 
are two real Hilbert spaces, 
$A_1$ is maximally monotone on $X$
and $A_2$ is maximally monotone on $Y$.
We also have a continuous linear operator
$L\colon X\to Y$, as well as 
$\sigma>0$ and $\tau>0$ such that 
$\sigma\tau\|L\|^2\leq 1$.
The goal is to find a point in\footnote{The operator $L^*A_2L$ is not necessarily maximally monotone. For a sufficient condition, 
see \cite[Corollary~25.6]{BC2017}.}
\begin{equation}
\zer(A_1 + L^*A_2L),
\end{equation}
which we assume to exist. 
We have
\begin{equation}
\label{e:230629a}
H = X \times Y 
\;\;\text{and}\;\;
M = \begin{bmatrix}
\frac{1}{\sigma}\Id_X & -L^*\\
-L & \frac{1}{\tau}\Id_Y
\end{bmatrix};
\end{equation}
hence, 
\begin{equation}
\label{e:230628b}
\|(x,y)\|_M^2 = \frac{1}{\sigma}\|x\|^2 - 2\scal{Lx}{y}+\frac{1}{\tau}\|y\|^2.
\end{equation}
Unfortunately, for Chambolle-Pock the operator $C$ in the factorization $CC^*$ is typically
\emph{not} explicitly 
available\footnote{However, see \cref{ss:moreCPFac} and \cref{ss:moreCPRR} below for some factorizations.}. 
Moreover, 
\begin{equation}
A = \begin{bmatrix}
A_1 & L^* \\
-L & A_2^{-1}
\end{bmatrix}. 
\end{equation}
Then $A$ is maximally monotone on $X\times Y$, 
because it is the sum of the maximally monotone operator $(x,y)\mapsto A_1x\times A_2^{-1}y$ and a skew linear operator. 
Next, 
\begin{subequations}
\label{e:zerACP}
\begin{align}
\zer A 
&= 
\menge{(x,y)}{0\in A_1x+L^*y \land 0\in -Lx+A_2^{-1}y}\\ 
&= 
\menge{(x,y)}{0\in A_1x+L^*y \land y\in A_2(Lx)}\\
&= 
\menge{(x,y)}{x\in\zer(A_1+L^*A_2L) \land y\in (-L^{-*}A_1x) \cap (A_2Lx)}\label{e:CPFixTset}\\
&= 
\menge{(x,y)}{y\in\zer(-L^*A_1^{-1}(-L) +A_2^{-1}) \land x\in (A_1^{-1}(-L^*y)) \cap (L^{-1}A_2^{-1}y)}\label{e:CPdualset}
\end{align}
\end{subequations}
where \cref{e:CPdualset} follows from elementary algebraic manipulations (see also \cite[Proposition~1]{EF}). 
This shows that when 
$(x,y)\in\zer A$, then $x$ is a primal solution while $y$ corresponds to a dual solution, i.e., $y$ satisfies $0\in L^*A_1^{-\ovee}Ly = A_2^{-1}y$. 
One verifies (see \cite[equation~(3.4)]{BCLN}) that 
\begin{equation}\label{e:MAinvCP}
(M+A)^{-1}\colon X\times Y\to X\times Y\colon 
\begin{bmatrix} x\\ y\end{bmatrix} \mapsto \begin{bmatrix}
\J{\sigma A_1}(\sigma x)\\\J{\tau A_2^{-1}}\big(2\tau L \J{\sigma A_1}(\sigma x)+\tau y\big)
\end{bmatrix}
\end{equation}
is indeed single-valued, full-domain, and Lipschitz continuous.
Hence (see \cite[equation~(3.3)]{BCLN}) 
\begin{align}
\label{e:TCP}
T(x,y) &= (M+A)^{-1}M(x,y) 
= \begin{bmatrix}
\J{\sigma A_1}(x-\sigma L^*y)\\
\J{\tau A_2^{-1}}\big(y+\tau L(2\J{\sigma A_1}(x-\sigma L^*y)-x) \big)
\end{bmatrix}.
\end{align}
Using the general 
inverse resolvent identity, (see, e.g., \cite[Proposition~23.20]{BC2017}), 
we know that 
$\J{\tau A_2^{-1}}=\Id-\tau\J{\frac{1}{\tau}A_2}\circ \frac{1}{\tau}\Id$.
Therefore, we can express 
\cref{e:TCP} also as
\begin{equation}\label{e:CPnoA2inv}
T(x,y) = 
\begin{bmatrix}
\J{\sigma A_1}(x-\sigma L^*y)\\[0.5em]
y+\tau L(2\J{\sigma A_1}(x-\sigma L^*y)-x)-\tau\J{\frac1\tau A_2}\big(\frac{1}{\tau}y+L(2\J{\sigma A_1}(x-\sigma L^*y)-x) \big)
\end{bmatrix}.
\end{equation}

\subsubsection{An example without strong convergence}
\label{sss:CPnotstrong}
Now suppose temporarily that 
$X=\ell^2$, that $B$, $s_0$ and $\bar{s}$ are as in \cref{p:GL}, 
that $A_1 = \tfrac{1}{\sigma}B$, 
and  that $A_2=0$. 
Then \cref{e:zerACP} and \cref{e:TCP} turn into
\begin{equation}
\zer A = \zer(B)\times\{0\}
\;\;\text{and}\;\;
T(x,y) = \begin{bmatrix}
\J{B}(x-\sigma L^*y)\\
0
\end{bmatrix}. 
\end{equation}
Suppose that furthermore $(\forall\kkk)$ $\lambda_k=1$ and 
$u_0 = (s_0,0)$. 
It then follows that 
$u_k = T^k(s_0,0) = (J_{B}^k(s_0),0)\weakly (\bar{s},0)$ 
but the convergence is not strong.

\subsubsection{Specialization to normal cones of linear subspaces}

\label{sss:CPtwosup}

We now assume that $U$ is closed linear subspace of $X$,
that $V$ is a closed linear subspace of $Y$, 
that $A_1=N_{U},A_2=N_{V}$.
Let $(x,y)\in X\times Y$.
Then $A_1x = U^\perp$ if $x\in U$, and $A_1x=\varnothing$ if 
$x\notin U$ and similarly for $N_Vy$. 
Note that $0=0+0\in U^\perp + L^*V^\perp$. 
It follows that 
$x\in \zer(A_1+L^*A_2L)$
$\Leftrightarrow$ 
$0\in A_1x+L^*A_2Lx$
$\Leftrightarrow$ 
$[x\in U \land Lx\in V]$ 
$\Leftrightarrow$ 
$x\in U \cap L^{-1}(V)$. 
We have shown 
$\zer(A_1+L^*A_2L)=U\cap L^{-1}(V)$. 
Now assume that $x\in \zer(A_1+L^*A_2L)=U\cap L^{-1}(V)$. 
Then
$y\in (-L^{-*}A_1x)\cap (A_2Lx)$
$\Leftrightarrow$ 
$[L^*(-y)\in A_1x \land y\in A_2Lx]$
$\Leftrightarrow$ 
$[-L^*y\in U^\perp \land y\in V^\perp]$
$\Leftrightarrow$ 
$[L^*y\in -U^\perp=U^\perp \land y\in V^\perp]$
$\Leftrightarrow$ 
$y\in L^{-*}(U^\perp) \cap V^\perp$.
Combining this with 
\cref{e:zerACP} yields 
\begin{equation}
\label{e:230628a}
\Fix T = \zer A = 
\big(U\cap L^{-1}(V)\big) \times 
\big(V^\perp \cap L^{-*}(U^\perp) \big).
\end{equation}
Now \cref{e:TCP} 
{and \cref{e:CPnoA2inv}} particularize to 
\begin{align}
\label{e:230629c}
T(x,y)
&= 
\begin{bmatrix}
P_U(x-\sigma L^*y)\\
P_{V^\perp}\big(y+\tau L(2P_{U}(x-\sigma L^*y)-x) \big)
\end{bmatrix}
\end{align}
and 
\begin{equation}
T(x,y) =\begin{bmatrix}
P_{U}(x-\sigma L^*y)\\[0.5em]
y+\tau L(2P_U(x-\sigma L^*y)-x)- P_{V}\big(y+\tau L(2P_{U}(x-\sigma L^*y)-x) \big)
\end{bmatrix}
\end{equation}
respectively. 
\begin{lemma}
\label{l:230628a}
Given $(x_0,y_0)\in X\times Y$, we have 
\begin{equation}
P_{\Fix T}^M(x_0,y_0)
=
\big(
P_{U\cap L^{-1}(V)}(x_0-\sigma L^*y_0),
P_{V^\perp\cap L^{-*}(U^\perp)}(y_0-\tau Lx_0)
\big).
\end{equation}
\end{lemma}
\begin{proof}
Using \cref{e:230628a}, we want to find the 
(unique by \cref{t:krista}\cref{t:krista3}) minimizer of the function 
$(x,y)\mapsto \|(x,y)-(x_0,y_0)\|_M$ subject to 
$x\in U\cap L^{-1}(V)$ and 
$y\in V^\perp \cap L^{-*}(U^\perp)$.
First, squaring $\|(x,y)-(x_0,y_0)\|_M$ and recalling 
\cref{e:230628b} yields
$\frac{1}{\sigma}\|x-x_0\|^2
-2\scal{L(x-x_0)}{y-y_0}+
\frac{1}{\tau}\|y-y_0\|^2$.
Second, expanding, discarding constant terms, and using $Lx\perp y$ results in 
$\frac{1}{\sigma}\|x\|^2
-\frac{2}{\sigma}\scal{x}{x_0}
+2\scal{Lx_0}{y}
+2\scal{Lx}{y_0}+
\frac{1}{\tau}\|y\|^2
-\frac{2}{\tau}\scal{y}{y_0}
=\frac{1}{\sigma}
\big(\|x\|^2-2\scal{x}{x_0-\sigma L^*y_0}
\big)
+\frac{1}{\tau}
\big(\|y\|^2-2\scal{y}{y_0-\tau Lx_0}
\big)
$.
Thirdly, completing the squares (which adds only constant terms here), we obtain 
$\frac{1}{\sigma}\|x-(x_0-\sigma L^*y_0)\|^2
+\frac{1}{\tau}\|y-(y_0-\tau Lx_0)\|^2$
and the proclaimed identity follows. 
\end{proof}

If $(u_k)_\kkk = (x_k,y_k)_\kkk$ is the PPP sequence for Chambolle-Pock, then 
the weak limit $u^*=(x^*,y^*)$ of $(Tu_k)_\kkk$
is given by 
\begin{equation}
\label{e:230628c}
(x^*,y^*)
=
\big(
P_{U\cap L^{-1}(V)}(x_0-\sigma L^*y_0),
P_{V^\perp\cap L^{-*}(U^\perp)}(y_0-\tau Lx_0)
\big)
\end{equation}
because of \cref{t:wconv}\cref{t:wconv2} and \cref{l:230628a}. 
Moreover, combining with \cref{t:sconv}, we obtain 
the \emph{strong convergence} result
\begin{equation}
\label{e:230628e}
P_{U}(x_k-\sigma L^*y_k)\to P_{U\cap L^{-1}(V)}(x_0)
\quad\text{when $y_0=0$ and $(\lambda_k)_\kkk\equiv\lambda$.}
\end{equation}

\begin{remark}
We note that \cref{e:230628c} beautifully 
generalizes \cref{e:230628d}, where
$Y=X$, $L=\Id$ and $\sigma=\tau=1$.
To the best of our knowledge, 
the limit formula \cref{e:230628c} appears to be new, 
even in the classical case where $M$ is positive definite. 
Moreover, \cref{e:230628e} is a nice way to compute algorithmically $P_{U\cap L^{-1}(V)}$ in the case when only $P_V$ is available but $P_{L^{-1}(V)}$ is not. 
\end{remark}

\subsubsection{Specializing \cref{sss:CPtwosup} even further to two lines in $\RR^2$}

\label{sss:2lines}

We now turn to a pleasing special case that allows
for the explicit computation of the spectral radius and the operator norm of $T$.
Keeping the setup of \cref{sss:CPtwosup}, 
we assume additionally that $Y=X$, $L=\Id$, and $\sigma=1/\tau$.
Then the operator $T$ from \cref{e:230629c} 
turns into 
\begin{align}
\label{e:230629d}
T 
&= 
\begin{bmatrix}
P_U & -\frac{1}{\tau}P_U \\
\tau P_{V^\perp}(P_U-P_{U^\perp}) & P_{V^\perp}(P_{U^\perp}-P_U)
\end{bmatrix}.
\end{align}
We now specialize even further to 
$X=\RR^2$ and $\theta\in\left[0,\pi\right[$, where we consider the two lines
\begin{equation}
U = \RR[1,0]{\tran}
\;\;\text{and}\;\;
V = \RR[\cos(\theta),\sin(\theta)]{\tran}
\end{equation}
which have an angle of $\theta$ and for 
which have projection matrices 
(see, e.g., \cite[Section~5]{BBCNPW})
\begin{equation}
P_U = \begin{bmatrix}
1&0\\0&0
\end{bmatrix}
\;\;\text{and}\;\;
P_V = 
\begin{bmatrix}
\cos^2(\theta)&\cos(\theta)\sin(\theta)\\
\cos(\theta)\sin(\theta)&\sin^2(\theta)
\end{bmatrix}. 
\end{equation}
It follows that \cref{e:230629d} turns into
\begin{equation}
\label{eq:Tsym}
T = 
\begin{bmatrix}
1 & 0 & -\frac{1}{\tau} & 0\\
0 & 0 & 0 & 0\\
\tau\,\sin^2 (\theta ) & \tau\,\cos \left(\theta \right)\,\sin \left(\theta \right) & -\sin ^2 (\theta ) & -\cos \left(\theta \right)\,\sin \left(\theta \right)\\
-\tau\,\cos (\theta )\,\sin (\theta) & -\tau\cos^2 (\theta ) & \cos \left(\theta \right)\,\sin(\theta) & \cos^2 (\theta)
\end{bmatrix}. 
\end{equation}
We know from the main results of Bredies et al.\ that 
$T$ has excellent properties with respect to 
the PPP algorithm.
For this particular $T$, we can quantify the key notions:

\begin{lemma}
\label{l:2lines}
The operator $T$ defined in \cref{eq:Tsym} satisfies the following: 
\begin{enumerate}
\item 
\label{l:2lines1}
{\rm {\bf (spectral radius)}}
$\rho(T)=|\cos(\theta)|$, which is \emph{independent} of $\tau$, and minimized with minimum value of $0$ when $\theta=\pi/2$ and maximized with maximum value of $1$ when $\theta=0$. 
\item 
\label{l:2lines2}
{\rm {\bf (operator norm)}}
We have 
\begin{equation}
\|T\| 
= 
\sqrt{1 + \frac{1+\tau^4+(1+\tau^2)\sqrt{1+\tau^4-2\tau^2\cos(2\theta)}}{2\tau^2}}.
\end{equation}

\item
\label{l:2lines3}
{\rm {\bf (bounds)}}
$\sqrt{2}\leq \sqrt{1+\max\{\tau^2,1/\tau^2\}} \leq \|T\| \leq \tau+{1}/{\tau}$,
and the lower bound is attained when $\theta=0$ and $\tau=1$ while the upper bound is attained when $\theta=\pi/2$.
\item 
\label{l:2lines4}
{\rm {\bf (Douglas-Rachford case)}}
When $\tau=1$, we have $\|T\|=\sqrt{2+2\sin(\theta)}$ which is minimized with minimum value $\sqrt{2}$ when $\theta=0$, and maximized with maximum value $2$ when $\theta=\pi/2$.  
\end{enumerate}
\end{lemma}
\begin{proof}
The matrix $T$ has 
the characteristic polynomial (in the variable $\zeta$):
\begin{equation}
\zeta^2\big(\zeta^2-2\cos^2(\theta)\zeta+\cos^2(\theta) \big);
\end{equation}
hence, the eigenvalues of $T$ are
$0,0,|\cos(\theta)|\big(|\cos(\theta)|\pm\mathrm{i}|\sin(\theta)|\big)$, with absolute values\\ $0,0,|\cos(\theta)|\sqrt{|\cos(\theta)|^2+|(\pm|\sin(\theta)|)|^2}=|\cos(\theta)|$. Thus we obtain that the
spectral radius of $T$ to be
\begin{equation}
\rho(T) = |\cos(\theta)|.
\end{equation}
This verifies \cref{l:2lines1}.

Next, the matrix $T{\tran}T$ has
the characteristic polynomial (in the variable $\zeta$)
\begin{equation}
\frac{\zeta^2}{\tau^2}\Big( \tau^2\zeta^2 -(1+\tau^2)^2\zeta +(1+\tau^2)^2\cos^2(\theta)\Big)
\end{equation}
and thus it has eigenvalues $0,0$ and 
\begin{subequations}
\begin{align}
&\hspace{-1cm} \frac{(1+\tau^2)^2\pm \sqrt{(1+\tau^2)^4-4\tau^2(1+\tau^2)^2\cos^2(\theta)}}{2\tau^2}\\
&=
\frac{(1+\tau^2)^2\pm(1+\tau^2)\sqrt{1+\tau^4+2\tau^2(1-2\cos^2(\theta))}}{2\tau^2}
\\
&=
\frac{(1+\tau^2)^2\pm(1+\tau^2)\sqrt{1+\tau^4-2\tau^2\cos(2\theta)}}{2\tau^2}
\\
&=
1 + \frac{1+\tau^4\pm(1+\tau^2)\sqrt{1+\tau^4-2\tau^2\cos(2\theta)}}{2\tau^2}.
\end{align}
\end{subequations}
Therefore,
\begin{align}
\|T\| 
&= 
\sqrt{1 + \frac{1+\tau^4+(1+\tau^2)\sqrt{1+\tau^4-2\tau^2\cos(2\theta)}}{2\tau^2}}
\end{align}
and we have verified \cref{l:2lines2}.

Turning to \cref{l:2lines3}, 
we see that for fixed $\tau$ that 
$\|T\|$ is smallest (resp.\ largest)
when $\theta=0$ (resp.\ $\theta=\pi/2$) in which case 
$\|T\|$ becomes
$
\sqrt{1+({1+\tau^4+ (1+\tau^2)|1-\tau^2|})/({2\tau^2})}$ (resp.  $\sqrt{1+({1+\tau^4+ (1+\tau^2)(1+\tau^2)})/({2\tau^2})}\big)
$
which further simplifies to $\sqrt{1+\max\{\tau^2,1/\tau^2\}}$  
(resp.\ $\tau+1/\tau$). 

Finally, \cref{l:2lines4} follows by
simplifying \cref{l:2lines2} with $\tau=1$. 
\end{proof}

\begin{remark}
\cref{l:2lines}\cref{l:2lines3} impressively shows that studying the operator $T$ from \cref{eq:Tsym} is outside the realm of fixed point theory of classical nonexpansive mappings. In such cases, the spectral radius is helpful (for more results in this direction, see \cite{BBCNPW2}).
We also point out that the operator $T$
from \cref{eq:Tsym} is firmly nonexpansive --- not with respect to the standard Hilbert space norm on $X^2$, but rather with respect to the seminorm induced by the preconditioner $M$
(see \cite[Lemma~2.6]{BCLN} for details).
\end{remark}

\subsubsection{Numerical experiment}

In this section, let 
$U$ be a closed affine subspace of $X$ and 
let $b\in Y$.
Let $(x,y)\in X\times Y$. 
We assume that $A_1 = N_U$ and $A_2 = N_{\{b\}}$. 
Then \cref{e:CPnoA2inv} turns into 
\begin{equation}\label{e:CPnoA2invb}
T(x,y) = 
\begin{bmatrix}
P_U(x-\sigma L^*y)\\[0.5em]
y+\tau L(2P_U(x-\sigma L^*y)-x)-\tau b
\end{bmatrix}.
\end{equation}
Now assume that $x\in \zer(A_1+L^*A_2L)=U\cap L^{-1}(b)$.
Then
$y\in (-L^{-*}A_1x)\cap (A_2Lx)$
$\Leftrightarrow$ 
$[L^*(-y)\in A_1x \land y\in A_2Lx]$
$\Leftrightarrow$ 
$[-L^*y\in (U-U)^\perp \land y\in (\{b\}-\{b\})^\perp]$
$\Leftrightarrow$ 
$[L^*y\in -(U-U)^\perp=(U-U)^\perp \land y\in Y]$
$\Leftrightarrow$ 
$y\in L^{-*}((U-U)^\perp)$.
It follows from \cref{e:zerACP} that 
\begin{equation}
\zer A = \Fix T = \big(U\cap L^{-1}(b)\big)
\times L^{-*}\big((U-U)^\perp \big). 
\end{equation}
Arguing similarly to the derivation of 
\cref{e:230628c} and \cref{e:230628e}, we obtain 
that the weak limit 
$u^* = (x^*,y^*)$ of 
the PPP sequence $(u_k)_\kkk = (x_k,y_k)_\kkk$ 
of Chambolle-Pock is given by 
\begin{equation}
\label{e:240425a}
(x^*,y^*) = 
\big(P_{U\cap L^{-1}(b)}(x_0-\sigma L^*y_0),
P_{L^{-*}((U-U)^\perp)}(y_0-\tau Lx_0)\big)
\end{equation}
and that 
\begin{equation}
\label{e:240425b}
P_U(x_k-\sigma L^*y_k) \to 
P_{U\cap L^{-1}(b)}(x_0) 
\quad \text{when $y_0=0$ and $(\lambda_k)_\kkk \equiv \lambda$.}
\end{equation}

We now illustrate \cref{e:240425b} numerically using 
a set up motivated by 
Computed Tomography (CT), which uses X-ray measurements to reconstruct cross-sectional body images (see, e.g., 
 \cite{Herman}). 
The resulting inverse problem we consider here amounts to solving 
$Lx=b$ where $L\in \RR^{6750\times 2500}$ and 
$b\in \RR^{6750}$: indeed, the matrix $L$ and the vector $b$ were generated from  
the $50\times 50$ Shepp-Logan phantom image \cite{SL}, reshaped to a vector 
$x^*\in \RR^{2500}$ 
using the \texttt{Matlab} 
AIR Tools II package\footnote{We 
used the command \texttt{paralleltomo(50,0:2:178,75)}. 
In fact, the experiments were performed in \texttt{GNU Octave} \cite{Octave}.} \cite{phantomcode}. 
In turn, the closed linear subspace $U$ was obtained by using 
the \emph{a priori} information that the first and last two columns of the phantom image must be black, i.e., 
they contain only zeros. 
Because the matrix $L^*L$ is smaller than $LL^*$, we compute 
$\|L\|$ via $\|L\|^2=\|L^*L\|$ and then set 
$\sigma := \tau := 0.99/\|L\|$. 
In \Cref{figure:phantom} we  present the reconstructed phantom images generated after $100$ and $10000$ iterations of 
Chambolle-Pock, with starting point 
$(x_0,y_0) = (0,0)$ and $\lambda_k\equiv 1$, 
along with the exact phantom.

\begin{figure}[ht]
  \centering
  \begin{tabular}{c c c c c}
   \includegraphics[scale=0.35]{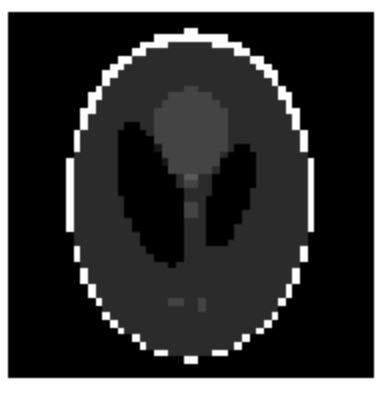}
   &
   &
   \includegraphics[scale=0.35]{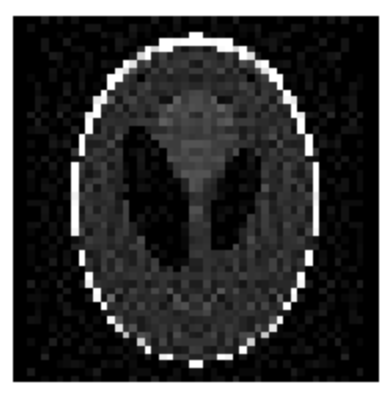}
    &  
    &
   \includegraphics[scale=0.35]{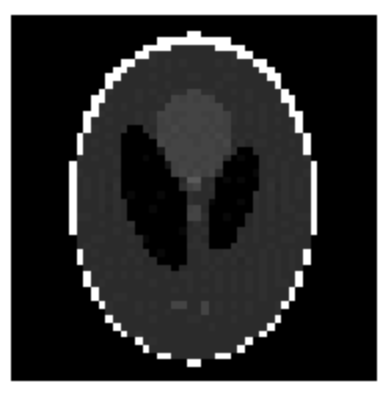}
   \\
   (a) Exact phantom image
   &
   &
   (b) Image generated by $x_{100}$
   &
  &
   (c) Image generated by $x_{10000}$
  \end{tabular}
  \caption{The exact phantom image and reconstructions generated by
  $x_{100}$ and $x_{10000}$.}
  \label{figure:phantom}
\end{figure}

\subsection{Ryu}

\subsubsection{General setup}

Let $X$ be a real Hilbert space, and 
$A_1,A_2,A_3$ be maximally monotone on $X$.
The goal is to find a point in 
\begin{equation}
 Z := \zer(A_1+A_2+A_3)
\end{equation}
which we assume to be nonempty. 
The three-operator splitting algorithm by Ryu \cite{Ryu} is designed to solve this problem. 
It was pointed out in \cite{BCN} that this fits the framework by Bredies et al. 
Indeed, assume that 
\begin{align}
\label{e:gymC}
H = X^5, \quad D = X^2, \quad\text{and}\quad 
C = \begin{bmatrix}
\Id & 0 \\
0 & \Id \\
-\Id & -\Id \\
\Id & 0 \\
0 & \Id
\end{bmatrix}. 
\end{align}
Then 
\begin{equation}
\label{e:gymC*}
C^* = 
\begin{bmatrix}
\Id & 0 & -\Id & \Id & 0 \\
0 & \Id & -\Id & 0 & \Id
\end{bmatrix}
\end{equation}
is clearly surjective and 
\begin{equation}
\label{e:gymM}
M = CC^* = 
\begin{bmatrix}
\Id & 0 &-\Id & \Id & 0\\
0 &\Id &-\Id & 0 & \Id\\
-\Id & -\Id & 2\Id & -\Id & -\Id \\
\Id & 0 & -\Id & \Id & 0 \\
0 & \Id &-\Id & 0 & \Id
\end{bmatrix}. 
\end{equation}
Now assume that 
\begin{equation}
A = 
\begin{bmatrix}
2A_1+\Id & 0 &\Id & -\Id & 0\\
-2\Id &2A_2+\Id &\Id & 0 & -\Id\\
-\Id & -\Id & 2A_3 & \Id & \Id \\
\Id & 0 & -\Id & 0 & 0 \\
0 & \Id &-\Id & 0 & 0 
\end{bmatrix}, 
\end{equation}
which is maximally monotone\footnote{Indeed, 
$A$ is the sum of the operator 
$(x_1,\ldots,x_5)\mapsto 2A_1x_1\times 2A_2x_2\times 2A_3x_3 
\times \{0\} \times \{0\}$, 
the gradient of $(x_1,\ldots,x_5)\mapsto \frac{1}{2}\norm{x_1-x_2}^2$, and a skew linear operator.} on $H$.
Next, given $(x_1,\ldots,x_5)\in H$, we have 
\begin{subequations}
\begin{align}
(x_1,\ldots,x_5)\in\zer A
&\Leftrightarrow 
\begin{cases}
0\in (2A_1+\Id)x_1+x_3-x_4\\
0\in-2x_1+(2A_2+\Id)x_2+x_3-x_5\\
0\in -x_1-x_2+2A_3x_3 +x_4+x_5\\
0=x_1-x_3\\
0=x_2-x_3 
\end{cases}\\
&\Leftrightarrow 
\begin{cases}
x = x_1 = x_2 = x_3\\
x_4 \in 2A_1x + 2x\\
x_5 \in 2A_2x\\
-x_4-x_5 \in 2A_3x-2x.
\end{cases}
\end{align}
\end{subequations}
This gives the description 
\begin{equation}
\label{e:Ryuzer}
\zer A = \menge{(z,z,z,x_4,x_5)\in H}{z\in Z\land\thalb x_4\in A_1z+z\land \thalb x_5 \in A_2z\land -\thalb x_4-\thalb x_5 \in A_3z-z}. 
\end{equation}
Next, one verifies that 
\begin{equation}
\label{e:gymbla}
(M+A)^{-1}:X^5\to X^5:\begin{bmatrix}
x_1\\x_2\\x_3\\x_4\\x_5
\end{bmatrix}
\mapsto
\begin{bmatrix}
y_1\\
y_2\\
y_3\\
y_4\\
y_5
\end{bmatrix}=
\begin{bmatrix}
\J{A_1}\big(\thalb x_1\big)\\[+1mm]
\J{A_2}\big(\thalb x_2+y_1\big)\\[+1mm]
\J{A_3}\big(\thalb x_3+y_1+y_2\big)\\[+1mm]
x_4-2y_1+2y_3\\[+1mm]
x_5-2y_2+2y_3
\end{bmatrix}, 
\end{equation}
which is clearly single-valued, full domain, and Lipschitz continuous.
Combining \cref{e:gymbla} and \cref{e:gymM}, we end 
up with 
\begin{subequations}
\begin{align}
\label{e:gymblu}
T = (M+A)^{-1}M:X^5 &\to X^5\\
\begin{bmatrix}
x_1\\x_2\\x_3\\x_4\\x_5
\end{bmatrix}
&\mapsto
\begin{bmatrix}
y_1\\
y_2\\
y_3\\
y_4\\
y_5
\end{bmatrix}=
\begin{bmatrix}
\J{A_1} \parens[\big]{\tfrac12 (x_1-x_3+x_4)}\\[+1mm]
\J{A_2} \parens[\big]{\tfrac12 (x_2-x_3+x_5)+y_1}\\[+1mm]
\J{A_3} \parens[\big]{\tfrac12 (-x_1-x_2+2x_3-x_4-x_5)+y_1+y_2}\\[+1mm]
x_1-x_3+x_4-2y_1+2y_3\\[+1mm]
x_2-x_3+x_5-2y_2+2y_3
\end{bmatrix}. 
\end{align}
\end{subequations}
Turning to $\Tt$, we verify that 
$\Tt = C^*(M+A)^{-1}C:D\to D$ is given by 
\begin{equation}
\label{e:gymblo}
\Tt \colon 
\begin{bmatrix}
w_1\\w_2
\end{bmatrix}
\mapsto
\begin{bmatrix}
w_1\\w_2
\end{bmatrix} 
+ 
\begin{bmatrix}
y_3-y_1\\y_3-y_2
\end{bmatrix}, 
\text{~~where}
\begin{bmatrix}
y_1\\
y_2\\
y_3
\end{bmatrix}=
\begin{bmatrix}
\J{A_1} \parens[\big]{\tfrac12 w_1}\\[+1mm]
\J{A_2} \parens[\big]{\tfrac12 w_2+z_1}\\[+1mm]
\J{A_3} \parens[\big]{\tfrac12(-w_1-w_2)+y_1+y_2}
\end{bmatrix}. 
\end{equation}
The operator $\Tt$ is a scaled version\footnote{If we denote the original Ryu operator by $T_{\rm R}$, then 
$\Tt(w) = 2T_{\rm R}(w/2)$. So, 
$\Fix \Tt = 2\Fix T_{\rm R}$ and $\Tt^k(w)=2T_{\rm R}^k(w/2)$.} of 
the original Ryu operator. 
Combining \cref{e:Ryuzer} and \cref{e:FixTtilde}, 
we obtain 
\begin{equation}
\label{e:RyuFTd}
\Fix\Tt =\\ \menge{(w_1,w_2)\in H}{z\in Z\land\thalb w_1\in A_1z+z\land \thalb w_2 \in A_2z\land -\thalb w_1-\thalb w_2 \in A_3z-z}. 
\end{equation}

\subsubsection{Specialization to normal cones of linear subspaces}
\label{sss:Ryu3sup}

We now assume that each $A_i = N_{U_i}$, 
where each $U_i$ is a closed linear subspace of $X$. 
Then 
$Z=\zer(A_1+A_2+A_3)=U_1\cap U_2 \cap U_3$. 
Then \cref{e:Ryuzer} turns into 
\begin{equation}
\label{e:RyuFixTDefn}
\zer A = \Fix T = \menge{(z,z,z,x_4,x_5)\in H}{z\in Z\land x_4\in U_1^\perp +2z\land x_5 \in  U_2^\perp \land x_4+x_5 \in U_3^\perp +2z}, 
\end{equation}
while \cref{e:RyuFTd} becomes
\begin{equation}\label{e:RyuFixTTDefn}
\Fix \Tt=
\menge{(w_1,w_2)\in D}{z\in Z\land w_1\in U_1^\perp +2z\land w_2 \in  U_2^\perp \land w_1+w_2 \in U_3^\perp +2z}.
\end{equation}

We now provide an alternative description of the
two fixed point sets.

\begin{lemma}[fixed point sets]
\label{lem:RyuFix}
Set 
$S := \menge{(z,z,z,2z,0)}{z\in Z=U_1\cap U_2\cap U_3}\subseteq X^5$ and 
$\Delta_2 :=\menge{(x,x)}{x\in X}\subseteq X^2$. 
Then\footnote{When $S_1\perp S_2$, we also write $S_1\oplus S_2$ for $S_1+S_2$ to stress the orthogonality.}
\begin{equation}
\label{lem:RyuFixi} 
\Fix T= S  \oplus \big(\{0\}^3\times \big( (U_1^\perp \times U_2^\perp) \cap (\Delta_2^\perp  + (\{0\}\times U_3^\perp))\big)\big)
\end{equation}
and 
\begin{equation}
\label{lem:RyuFixii} 
\Fix \Tt 
= (Z\times \{0\}) \oplus \big((U_1^\perp \times U_2^\perp)\cap (\Delta_2^\perp  + (\{0\}\times U_3^\perp))\big). 
\end{equation}
\end{lemma}
\begin{proof}
Recall that $\Delta_2^\perp = \menge{(x,-x)}{x\in X}$. 
The identity \cref{lem:RyuFixi} is  
a reformulation of \cref{e:RyuFixTDefn}. 
To obtain \cref{lem:RyuFixii}, combine 
\cref{e:RyuFixTTDefn}
with the fact that $2Z=Z$. 
The orthogonality statements are a consequence of 
$Z\subseteq U_1\perp U_1^\perp$.
\end{proof}

\begin{corollary}[projections]
Set $E := (U_1^\perp \times U_2^\perp)\cap (\Delta_2^\perp  + (\{0\}\times U_3^\perp))$, 
where $\Delta_2$ is as in 
\cref{lem:RyuFix}. 
Then 
\begin{equation}
\label{lem:RyuFixiii}
(\forall w=(w_1,w_2)\in D=X^2)\quad
P_{\Fix \Tt}w=\big(P_Z(w_1),0\big)+P_E(w).
\end{equation}
Now let 
$u = (u_1,u_2,u_3,u_4,u_5)\in H = X^5$, 
set $w := C^*u$, and 
$w^*:=P_{\Fix \Tt}(w)$. 
Then 
\begin{equation}
\label{lem:RyuFixiv}
P_{\Fix T}^M(u)=\big(
\thalb P_Z w_1,
\thalb P_Z w_1,
\thalb P_Z w_1, 
w_1^\ast, 
w_2^\ast
\big)
\end{equation}
\end{corollary}
\begin{proof}
The identity \cref{lem:RyuFixiii} 
follows directly from \cref{lem:RyuFixii}. 
We now tackle 
\cref{lem:RyuFixiv}. 
From \cref{t:krista}\cref{t:krista1},
we get
\begin{equation}\label{e:PMFixTwstar}
P^M_{\Fix T}(u)
= (M+A)^{-1}CP_{\Fix\Tt}(C^* u)
= (M+A)^{-1}CP_{\Fix\Tt}(w)
= (M+A)^{-1}Cw^*. 
\end{equation}
By \cref{lem:RyuFixiii}, 
$w^*= (z+v_1,v_2)=(z+v_1,-v_1+v_3)$, 
where $z:=P_Z(w_1)$ and each $v_i\in U_i^\perp$.
Hence $Cw^* = 
(z+v_1,v_2,-z-v_3,z+v_1,v_2)
$ using \cref{e:gymC}. 
In view of \cref{e:gymbla}, we obtain 
\begin{align}
(M+A)^{-1}Cw^* 
&= (M+A)^{-1}\begin{bmatrix}
z+v_1\\
v_2\\
-z-v_3\\
z+v_1\\
v_2
\end{bmatrix}
= 
\begin{bmatrix}[1.3]
P_{U_1}\big(\frac12(z+v_1)\big)=\frac12 z\\
P_{U_2}\big(\frac12 v_2 +\frac12 z\big)=\frac12 z\\
P_{U_3}\big(-\frac12(z+v_3)+z\big)=\frac12 z\\
z+v_1\\
v_2
\end{bmatrix}
\end{align}
and we are done. 
\end{proof}

Of course if $(u_k)_\kkk$ is the PPP sequence for Ryu, then the weak limit of $(Tu_k)_\kkk$ is $P^M_{\Fix T}(u_0)$ due to \crefpart{t:wconv}{t:wconv3}. This is given by \cref{lem:RyuFixiv} with $u$ replaced by the starting point $u_0$. 
The weak limit of the rPPP sequence $(w_k)_\kkk$ 
is given by \cref{lem:RyuFixiii} (with $w$ replaced by the starting point $w_0$) --- this formula 
was observed already in \cite[Lemma~4.1]{BSW}.

\subsection{Malitsky-Tam}

\subsubsection{General setup}

For $n\geq 3$, we consider $n$ maximally monotone operators $A_1,\dots,A_n$ on the Hilbert space $X$. The goal becomes to find a point in
\begin{equation}
Z:=\zer(A_1+\dots+A_n)
\end{equation}
which we assume to be nonempty. The splitting algorithm by Malitsky and Tam \cite{MT} can deal with this problem and also fits the structure described in \cite{BCN}. Assume that
\begin{equation}
H=X^{2n-1},\;\; D=X^{n-1}, \;\;\text{and} \;\;C=\left[\begin{array}{rrrrr}
\Id\\-\Id&\Id\\&-\Id&\ddots&\\&&\ddots&\Id\\&&&-\Id\\\hline\\[-3mm]
\Id\\&\Id\\&&\ddots\\&&&\Id
\end{array}\right]\colon D\to H.
\end{equation}
Then
\begin{equation}
C^\ast=
\left[\begin{array}{rrlr|rrr}
\Id&-\Id&&&\Id\\&\ddots&\ddots&&&\ddots\\&&\Id&-\Id&&&\Id
\end{array}\right]
\colon H\to D \colon
\begin{bmatrix}
x_1\\\vdots\\x_n\\v_1\\\vdots\\v_{n-1}
\end{bmatrix}
\mapsto
\begin{bmatrix}
x_1-x_2+v_1\\
\vdots\\
x_{n-1}-x_{n}+v_{n-1}
\end{bmatrix}
\end{equation}
is surjective.
Using $C$ and $C^\ast$, we compute
\begin{equation}\label{e:MTM}
  M=\left[\begin{array}{rrrrrr|rrrrrrrr}
    \Id &-\Id   &       &       &       &       & \Id   &      &            \\
   -\Id & 2\Id  &-\Id   &       &       &       &-\Id   &\Id   &            \\
        &\rddots&\rddots&\rddots&       &       &       &-\Id  &\ddots      \\
        &       &-\Id   & 2\Id  &-\Id   &       &       &      &\ddots&\Id\\[0.5em]
        &       &       &-\Id   & \Id   &       &       &      &      &-\Id\\\hline
    \Id &-\Id   &       &       &       &       &\Id                        \\
        &\Id    &-\Id   &       &       &       &       &\phantom{-}\Id\\
        &       &\rddots&\rddots&       &       &       &      &\ddots\\
        &       &       & \Id   &-\Id   &       &       &      &       & \Id
   \end{array}\right].
\end{equation}
Finally, we assume that 
\begin{equation}
A=\left[\begin{array}{rrrcl|rrrlr}
    2A_1+\Id& \Id   &       &        &        &-\Id &      &\\
       -\Id & 2A_2  &\ddots &        &        & \Id &-\Id  &  \\
            &-\Id   &\ddots &\Id     &        &     &\Id   &\ddots  \\
            &       &\ddots &2A_{n-1}&\Id     &     &      &\ddots& &    -\Id \\[0.5em]
     -2\Id  &       &       & -\Id   &2A_n+\Id&     &      &  &  & \Id\\\hline
     \Id    &-\Id   &       &        &        &     &                   \\
            &\Id    &-\Id   &        &        &     &      &\makebox(-30,-20)0\\
            &       &\rddots&\rddots &        &     &&\\ 
            &       &       &\Id\phantom{-}   &-\Id &      & & & &
    \end{array}\right],
\end{equation}
which is maximally monotone because 
it can be written as the sum of
the maximally monotone operator 
$(x_1,\ldots,x_n,v_1,\ldots,v_{n-1})\mapsto 
2A_1x_1\times \cdots \times 2A_nx_n \times\{0\}^{n-1}$,
a linear monotone operator that is the gradient of a convex function, 
and a skew linear operator. 
Next, given $u:=(x_1,\dots, x_n,v_1,\dots,v_{n-1})\in H$, we have
\begin{subequations}
\begin{align}
u\in \zer A&\Leftrightarrow
\left\{\begin{array}{rlr}
0&\in (2A_1+\Id)x_1+x_2-v_1\\
0&\in- x_{i-1}+2A_ix_i+x_{i+1}+v_{i-1}-v_i&\text{ for all }  i\in\{2,\dots,n-1\}\\
0&\in -2x_1-x_{n-1}+(2A_n+\Id)x_n +v_{n-1}\\
0&=x_j-x_{j+1}&\text{ for all }j\in\{1,\dots,n-1\} 
\end{array}\right.\\
&\Leftrightarrow 
\begin{cases}
z=x_1=\dots=x_n\\
v_1\in 2(A_1+\Id)z\\
v_i\in 2A_iz+v_{i-1}&\text{ for all } i\in\{2,\dots,n-1\}\\
v_{n-1}\in 2(\Id-A_n)z\cap (2A_{n-1}z+v_{n-2}).
\end{cases}
\end{align}
\end{subequations}
This gives us
\begin{alignat}{3}
\zer A=& \big\{(z,\dots,z,v_1,\dots,v_{n-1})\in H \mid  &&  z\in \zer(A_1+\dots+A_n)\nonumber\land v_1\in 2(A_1+\Id)z \nonumber\\
 & &&(\forall i\in \{2,\dots,n-2\})\; v_i-v_{i-1} \in 2A_iz\nonumber\\
 & && \land v_{n-1}\in 2(\Id-A_n)z\cap (2A_{n-1}z+v_{n-2}) \big\}.\label{e:MTzerA}
\end{alignat}
Next, we verify that
\begin{equation}\label{e:MTMAinv}
(M+A)^{-1}:X^{2n-1}\to X^{2n-1}:
\begin{bmatrix}
x_1\\\vdots\\x_i\\\vdots\\x_n\\
v_1\\\vdots\\v_{n-1}
\end{bmatrix}
\mapsto
\begin{bmatrix}
y_1\\\vdots\\y_i\\\vdots\\y_n\\
w_1\\\vdots\\w_{n-1}
\end{bmatrix}
=
\begin{bmatrix}
\J{A_1} \parens[\big]{\frac12 x_1}\\\vdots\\
\J{A_i} \parens[\big]{\frac12 x_i+y_{i-1}}\\\vdots\\
\J{A_n} \parens[\big]{\frac12 x_n+y_1+y_{n-1}}\\
v_1-2y_1+2y_2\\
\vdots\\
v_{n-1}-2y_{n-1}+2y_n
\end{bmatrix}
\end{equation}
which is clearly single-valued, full domain, and Lipschitz continuous. Using the fact that $T=(M+A)^{-1}M$ along with \cref{e:MTMAinv} and \cref{e:MTM}, we get
\begin{equation}
T\colon H\to H \colon 
\begin{bmatrix}
x_1\\\vdots\\x_i\\\vdots\\x_n\\
v_1\\\vdots\\v_{n-1}
\end{bmatrix}
\mapsto
\begin{bmatrix}
y_1\\\vdots\\y_i\\\vdots\\y_n\\
w_1\\\vdots\\w_{n-1}
\end{bmatrix}
=
\begin{bmatrix}
\J{A_1} \parens[\big]{\frac12 (x_1-x_2+v_1)}\\
\vdots\\
\J{A_i}\parens[\big]{\frac12(-x_{i-1}+2x_i-x_{i+1}-v_{i-1}+v_i)+y_{i-1}}\\
\vdots\\
\J{A_n}\parens[\big]{\frac12(-x_{n-1}+x_n-v_{n-1})+y_1+y_{n-1}}\\
x_1-x_2+v_1-2y_1+2y_2\\
\vdots\\
x_n-x_{n-1}+v_{n-1}-2y_{n-1}+2y_n
\end{bmatrix}
\end{equation}
while 
\begin{align}
\Tt(w)&=\begin{bmatrix}
w_1\\
\vdots\\
w_{n-1}
\end{bmatrix}
+\begin{bmatrix}
z_2-z_1\\
\vdots\\
z_n-z_{n-1}
\end{bmatrix}
\text{ where }
\begin{bmatrix}
z_1\\
\vdots\\
z_i\\
\vdots\\
z_n
\end{bmatrix}=
\begin{bmatrix}
\J{ A_1}\parens[\big]{\frac12 w_1}\\
\vdots\\
\J{ A_i}\parens[\big]{\frac12 (-w_{i-1}+w_i)+z_{i-1}}\\
\vdots\\
\J{ A_n}\parens[\big]{\frac12 (-w_{n-1})+z_1+z_{n-1}}
\end{bmatrix}. 
\end{align}
Here, $\Tt$ becomes a scaled version of the Malitsky-Tam operator from \cite[Algorithm 1]{MT}. 
Finally, for $w=(w_1,\dots, w_{n-1})$, we have 
\begin{subequations}
\begin{align}
\Fix\Tt&=\big\{ w\in X^{n-1} \mid\; 
\begin{aligned}[t]
& z\in X\land w_1\in 2(\Id +A_1)z\\
&\land ( \forall i\in \{2,\dots,n-2\}) w_i- w_{i-1}\in 2A_i z \\
& \land w_{n-1}\in 2(\Id-A_n)z\cap (2A_{n-1}+w_{n-2})\big\}
\end{aligned}\label{e:240404a}\\
&=\big\{w\in X^{n-1} \mid 
\begin{aligned}[t]
& z\;\in \zer(A_1+\dots+A_n)\land w_1\in 2(\Id +A_1)z\\
& \land  (\forall i\in \{2,\dots,n-2\}) w_i- w_{i-1}\in 2A_i z \\
& \land w_{n-1}\in 2(\Id-A_n)z\cap (2A_{n-1}z+w_{n-2})\big\}.\label{e:240404b}
\end{aligned}
\end{align}
\end{subequations}
To obtain in \cref{e:240404b} that $z\in \zer(A_1+\dots+A_n)$, we use the telescoping property on the statements for $w_i$ in \cref{e:240404a}.

\subsubsection{Specialization to normal cones of linear subspaces}

We now assume that $U_i$ is a closed linear subspace of $X$ and that $A_i=N_{U_i}$ for  each $i\in\{1,\dots, n\}$.
Then $Z=\zer(A_1+\dots+A_n)=U_1\cap \dots \cap U_n$
and \cref{e:MTzerA} yields 
\begin{alignat}{2}
\Fix T&=\big\{(x,\dots,x,v_1,\dots,v_{n-1}) \mid\;  &&  x\in Z \land v_1-2x \in U_1^\perp\notag\\
& &&\land ( \forall i\in \{2,\dots,n-2\}) v_i \in U_i^\perp+v_{i-1} \nonumber\\
& && \land v_{n-1}\in (U_n^\perp+2x)\cap (U_{n-1}^\perp+v_{n-2}) \big\}\label{e:MTzerAlin}
\end{alignat}
and 
\begin{alignat}{2}\label{e:MTFixTTlin}
\Fix\Tt&=\big\{w\in X^{n-1}\mid &&\; z\in Z \land w_1\in U_1^\perp +2z\land ( \forall i\in \{2,\dots,n-2\}) w_i\in U_i^\perp + w_{i-1}\nonumber\\
& && \land w_{n-1}\in (U_n^\perp+ 2z)\cap ( U_{n-1}^\perp+w_{n-2})\big\}.
\end{alignat}

\begin{lemma}\label{lem:MTFix}
Set $\Delta=\menge{(x,\dots,x)\in X^{n-1}}{ x\in X}$ and define
\begin{subequations}
\begin{align}
E&:=\ran\Psi\cap(X^{n-2}\times U_n^\perp)\\
&\subseteq U_1^\perp \times \dots \times (U_1^\perp+\cdots+ U_{n-2}^\perp)\times ((U_1^\perp+\cdots+ U_{n-1}^\perp)\cap U_n^\perp)\label{e:EinUiperp}
\end{align}
\end{subequations}
such that
\begin{align}
\Psi:&= U_1^\perp \times \dots \times U_n^\perp \to X^{n-1}\colon
(y_1,\ldots, y_n)\mapsto (y_1,y_1+y_2,\dots, y_1+\dots+y_{n-1}).
\end{align}
Then
\begin{equation}\label{e:MTFixTlinset}
\Fix T=\menge{(\underbrace{x,\dots,x}_{n \text{\;copies}},2x,\dots,2x)\in X^{2n-1}}{x\in Z}\oplus (\{0\}^n\times E),
\end{equation} 
and
\begin{equation}\label{e:MTFixTtildelinset}
\Fix \Tt=\Delta_{Z^{n-1}}\oplus E,
\end{equation}
where $\Delta_{Z^{n-1}}:=\Delta\cap Z^{n-1}$.
\end{lemma}

\begin{proof}
Using \cref{e:MTzerAlin}, we know that a 
generic point in $\Fix T$ is represented by 
\begin{equation}
\begin{bmatrix}
x\\
\vdots\\
x\\
2x+y_1\\
\vdots\\
2x+\sum_{j=1}^{n-2} y_{j} \\
2x+\sum_{j=1}^{n-1} y_j 
\end{bmatrix}=
\begin{bmatrix}
x\\
\vdots\\x\\2x+y_1\\
\vdots\\
2x+\sum_{j=1}^{n-2}y_{j}\\
2x+y_n
\end{bmatrix},
\end{equation}
where $y_i\in U_i^\perp$ for each $i\in \{1,\dots,n\}$.
This yields \cref{e:MTFixTlinset}. The orthogonality in \cref{e:MTFixTlinset} follows from \cref{e:EinUiperp} and the fact that $Z^\perp=\overline{(U_1^\perp+\dots+U_n^\perp)}$. The representation for \cref{e:MTFixTtildelinset} can be obtained similarly from \cref{e:MTFixTTlin}. Finally, the orthogonality in \cref{e:MTFixTtildelinset} follows as 
$\Delta_{Z^{n-1}}\subseteq Z^{n-1} \perp \overline{U_1^\perp+\dots+U_n^\perp}^{n-1}\supseteq E$.
\end{proof}

\begin{corollary}
Using the definition of $E$ from \cref{lem:MTFix}, for $w\in X^{n-1}$,
\begin{equation}\label{e:MTPFixTTilde}
P_{\Fix \Tt}w=\big(P_Z(\wbar),\dots,P_Z(\wbar)\big)+P_E(w)
\end{equation}
where $\wbar:=\frac1{n-1} \sum_{i=1}^{n-1}w_i$.
Now let 
$u\in H=X^{2n-1}$, set $w=C^\ast u$ and 
$w^\ast=P_{\Fix \Tt}w$. Then
\begin{align}\label{e:MTPFixT}
P_{\Fix T}^M(u)&=\big( 
\thalb P_Z \wbar,
\ldots, 
\thalb P_Z \wbar, 
w_1^\ast, 
\ldots, 
w_{n-1}^\ast\big). 
\end{align}
\end{corollary}

\begin{proof}
The identity \cref{e:MTPFixTTilde} can be obtained from \cref{e:MTFixTtildelinset} after using the fact that $P_{Z^{n-1}\cap \Delta}=P_{Z^{n-1}}P_{\Delta}$ (see \cite[Proof of Lemma~4.3]{BSW}). Note that the projection onto $\Delta$ can be interpreted as reproducing the average of the argument a suitable number of times. For \cref{e:MTPFixT}, we observe through \cref{e:MTPFixTTilde} that $w^\ast_k=z+\sum_{j=1}^k v_j$ for $k\in \{1,\dots,n-1\}$ for some $v_i\in U_i^\perp$ and $z=P_Z(\wbar)$. Therefore using \cref{e:PMFixTwstar} again, but this time with \cref{e:MTMAinv}, we get
\begin{equation}
P_{\Fix T}^M(u)=(M+A)^{-1}
\begin{bmatrix}
z+v_1\\
v_2\\
\vdots\\
v_{n-1}\\
-v_n-z\\
w_1^\ast\\
\vdots\\
w_{n-1}^\ast
\end{bmatrix}
=\begin{bmatrix}
P_{U_1}(\frac12(z+v_1))=\frac12z\\
\vdots\\
P_{U_i}(\frac12 v_i+\frac12z)=\frac12z\\
\vdots\\
P_{U_n}(-\frac12 (v_n+z) +z)=\frac12z\\
w_1^\ast\\
\vdots\\
w_{n-1}^\ast
\end{bmatrix}
\end{equation}
and we are done. 
\end{proof}

\section{More on Chambolle-Pock}

In this section, we provide some additional
insights about the Chambolle-Pock framework
discussed in \cref{ss:CP}.

\label{s:moreCP}

\subsection{Factoring $M$}

\label{ss:moreCPFac}

In this section, we discuss how to find
a factorization of $M$ into $CC^*$ using Linear Algebra techniques. 
Suppose that $X=\RR^n$, $Y=\RR^m$,
and $L\in\RR^{m\times n}$.
Recall from \cref{e:230629a} that
\begin{equation}
\label{e:240928b}
M = \begin{bmatrix}
\frac{1}{\sigma}\Id_X & -L^* \\
-L & \frac{1}{\tau}\Id_Y
\end{bmatrix}\in\RR^{(n+m)\times(n+m)}. 
\end{equation}
Then $M\succeq 0$ 
$\Leftrightarrow$
$\sigma\tau\|L\|^2 \leq 1$ 
which we assume henceforth. 
The following result, 
which is easily verified directly, 
provides a factorization of $M$ into $CC^*$: 

\begin{lemma}
Suppose $Z\in\RR^{m\times m}$ satisfies\footnote{For instance, $Z$ could arise from a
Cholesky factorization of (or as the square root of) $\Id_Y - \sigma\tau LL^*$. A referee pointed out that the result can be generalized to Hilbert spaces.   Indeed, the proof of the operator version works exactly the same: use the form 
given in \cref{e:stefdent} to compute $CC^*$, then simplify the result with \cref{e:240928a} to finally obtain $M$ from \cref{e:240928b}.}
\begin{equation}
\label{e:240928a}
ZZ^* = \Id_Y - \sigma\tau LL^*.
\end{equation}
Then 
\begin{equation}
\label{e:stefdent}
C := \begin{bmatrix}
\frac{1}{\sqrt{\sigma}}\Id_X & 0 \\
-\sqrt{\sigma}L & \frac{1}{\sqrt{\tau}}Z
\end{bmatrix}
\end{equation}
factors $M$ into $CC^*$. 
\end{lemma}

\begin{example}
If $m=n=1$, 
$\sigma=\tau=1$, and $L=[\lambda]$, where 
$|\lambda|\leq 1$, then 
\begin{equation}
C = \begin{bmatrix}
1 & 0\\
-\lambda & \sqrt{1-\lambda^2}
\end{bmatrix}
\end{equation}
factors $M$ into $CC^*$. 
\end{example}

\begin{example}
If $\sigma\tau LL^*=\Id_Y$, 
then $Z=0$ and we can simplify and reduce \cref{e:stefdent} further to 
\begin{equation}
C := \begin{bmatrix}
\frac{1}{\sqrt{\sigma}}\Id_X\\
-\sqrt{\sigma}L
\end{bmatrix}\in\RR^{(n+m)\times n}, 
\end{equation}
which factors $M$ into $CC^*$. 
\end{example}

\subsection{An example on the real line}

\label{ss:moreCPRR}

Consider the general setup in \cref{sss:CPgeneral}.
We now specialize this further to the case 
when $X=Y=\RR$ and $L=\lambda\Id$, where 
$|\lambda|\leq 1$. 
Clearly, $\|L\|=|\lambda|$. 
We also assume that $\sigma=\tau=1$.
Then $\sigma\tau\|L\|^2\leq 1$ and the preconditioner matrix (see \cref{e:230629a}) is 
\begin{equation}
\label{e:230623a}
M_\lambda :=  \begin{bmatrix}
1 & -\lambda \\
-\lambda & 1
\end{bmatrix}.
\end{equation}
The matrix $M_\lambda$ acts on $H = \RR^2$, 
and it is indeed positive semidefinite, with eigenvalues 
$1\pm\lambda$ and 
\begin{equation}
\label{e:230623b}
\|M_\lambda\| = 1+|\lambda|.
\end{equation}
If $|\lambda|<1$, then 
\begin{equation}
\label{e:230623c}
M_\lambda^{-1} = 
\frac{1}{1-\lambda^2}
\begin{bmatrix}
1 & \lambda \\
\lambda & 1
\end{bmatrix}
\;\;\text{has eigenvalues $\frac{1}{1\pm\lambda}$,
\;\;and so \;\; $\|M_{\lambda}^{-1}\|=\frac{1}{1-|\lambda|}$}. 
\end{equation}
To find a factorization $M_\lambda=C_\lambda C_\lambda^*$, we follow the recipe given in \cite[Proof of Proposition~2.3]{BCLN}, which starts by
determining the principal square root of $M_{\lambda}$.
Indeed, one directly verifies that
\begin{equation}
\label{e:230623d}
S_\lambda := 
\frac{1}{2}\begin{bmatrix}
\sqrt{1-\lambda}+\sqrt{1+\lambda} &
\sqrt{1-\lambda}-\sqrt{1+\lambda} \\[+2mm]
\sqrt{1-\lambda}-\sqrt{1+\lambda} &
\sqrt{1-\lambda}+\sqrt{1+\lambda} 
\end{bmatrix} = \sqrt{M_\lambda}
\end{equation}
and that $S_\lambda$ has eigenvalues 
$\sqrt{1+\lambda},\sqrt{1-\lambda}$, 
and hence $\|S_\lambda\|=\sqrt{1+|\lambda|}$. 
Note that we have the equivalences 
$S_\lambda$ is invertible
$\Leftrightarrow$ $M_{\lambda}$ is invertible 
$\Leftrightarrow$ 
$|\lambda|<1$, in which case 
\begin{equation}
\label{e:230629b}
 S_\lambda^{-1} = 
\frac{1}{2}\begin{bmatrix}
\frac{1}{\sqrt{1-\lambda}}+
\frac{1}{\sqrt{1+\lambda}} &
\frac{1}{\sqrt{1-\lambda}}-
\frac{1}{\sqrt{1+\lambda}} \\[+2mm]
\frac{1}{\sqrt{1-\lambda}}-
\frac{1}{\sqrt{1+\lambda}} &
\frac{1}{\sqrt{1-\lambda}}
+\frac{1}{\sqrt{1+\lambda}} 
\end{bmatrix},
\end{equation}
$D = \cran S_\lambda = \RR^2 = H$, 
$S_{\lambda}^{-1}$ has eigenvalues 
$1/\sqrt{1+\lambda},1/\sqrt{1-\lambda}$, 
and $\|S_\lambda^{-1}\| = 1/\sqrt{1-|\lambda|}$. 
In addition,
\begin{equation}
S_{\pm 1} = 
\frac{1}{\sqrt{2}}\begin{bmatrix}
1 &
\mp 1 \\[+2mm]
\mp 1 &
1 
\end{bmatrix},
\end{equation}
and here $D = \cran S_{\pm 1} = \RR[1,\mp 1]{\tran}$ is a proper subspace of $H$. 

Let us sum up the discussion as follows.

\begin{lemma}
\label{l:230629a}
Using the square root approach, we may factor 
$M_\lambda$ into 
$M_\lambda = C_\lambda C_\lambda^*$ with 
\begin{equation}
C_\lambda:=\frac{1}{2}\begin{bmatrix}
\sqrt{1-\lambda}+\sqrt{1+\lambda} &
\sqrt{1-\lambda}-\sqrt{1+\lambda} \\[+2mm]
\sqrt{1-\lambda}-\sqrt{1+\lambda} &
\sqrt{1-\lambda}+\sqrt{1+\lambda} 
\end{bmatrix},
\end{equation}
where $C_\lambda = C_\lambda^*$. 
If $|\lambda|=1$, then we may also factor $M_{\pm 1}$ into $C_{\pm 1}C_{\pm 1}^*$, where
\begin{equation}
C_{\pm 1} = \begin{bmatrix} 1 \\ \mp 1\end{bmatrix}. 
\end{equation}
\end{lemma}

In fact, the verification of \cref{l:230629a} is 
\emph{algebraic in nature} and one may directly also verify the following operator variants 
of \cref{e:230623d} and \cref{e:230629b}:

\begin{lemma}
\label{l:240929a}
Suppose that $Y=X$ and let $L\colon X\to X$ be
symmetric and positive semidefinite with 
$\|L\|\leq 1$. 
Then the principal square root of the preconditioning operator 
\begin{equation}
M = \begin{bmatrix} \Id & -L \\ -L &\Id \end{bmatrix}
\end{equation}
is 
\begin{equation}
\sqrt{M} = 
\frac{1}{2}\begin{bmatrix}
\sqrt{\Id-L}+\sqrt{\Id+L} &
\sqrt{\Id-L}-\sqrt{\Id+L} \\[+2mm]
\sqrt{\Id-L}-\sqrt{\Id+L} &
\sqrt{\Id-L}+\sqrt{\Id+L} 
\end{bmatrix}. 
\end{equation}
If $\|L\|<1$, then $\Id\pm L$ are invertible and positive definite; moreover, 
\begin{equation}
\sqrt{M}^{-1} = 
\frac{1}{2}\begin{bmatrix}
\sqrt{(\Id-L)^{-1}}+\sqrt{(\Id+L)^{-1}} &
\sqrt{(\Id-L)^{-1}}-\sqrt{(\Id+L)^{-1}} \\[+2mm]
\sqrt{(\Id-L)^{-1}}-\sqrt{(\Id+L)^{-1}} &
\sqrt{(\Id-L)^{-1}}+\sqrt{(\Id+L)^{-1}} 
\end{bmatrix}. 
\end{equation}
\end{lemma}

The following generalization of \cref{l:240929a} 
was prompted by comments of a reviewer.

\begin{lemma}
\label{l:240929}
Suppose that $X=\RR^n$, $Y=\RR^m$, and $L\colon X\to Y$, i.e., 
$L\in\RR^{m\times n}$, satisfies $\|L\|\leq 1$. 
Set $S := \sqrt{L^*L}$, $U := LS^\dagger$, and $T := \sqrt{LL^*}$
which give rise to
the canonical polar decompositions of $L$ and $L^*$: 
\begin{equation}
L = US \quad\text{and} \quad L^* = U^*T.
\end{equation}
Set $A := \arcsin(S)$ and  $B := \arcsin(T)$. 
Then the principal square root of 
\begin{equation}
M := 
\begin{bmatrix}
\Id_X & -L^*\\
-L & \Id_Y
\end{bmatrix}
\end{equation}
is 
\begin{subequations}
\begin{align}
\sqrt{M} &= \begin{bmatrix}
\cos(A/2) & -U^*\sin(B/2)\\[+2mm]
-U\sin(A/2) & \cos(B/2)
\end{bmatrix} \label{e:wowtrig}\\[+4mm]
&= \frac{1}{2}
\begin{bmatrix}
\sqrt{\Id_X-\sqrt{L^*L}} + \sqrt{\Id_X+\sqrt{L^*L}} 
   &\;\; U^*\bigg(\sqrt{\Id_Y-\sqrt{LL^*}} - \sqrt{\Id_Y+\sqrt{LL^*}} \bigg)
\\[+4mm]
U\bigg(\sqrt{\Id_X-\sqrt{L^*L}} - \sqrt{\Id_X+\sqrt{L^*L}} \bigg) 
   &\;\; \sqrt{\Id_Y-\sqrt{LL^*}} + \sqrt{\Id_Y+\sqrt{LL^*}} 
\end{bmatrix}. \label{e:wowroot}
\end{align}
\end{subequations}
\end{lemma}
\begin{proof}
Because this result is not needed elsewhere, 
we only sketch the proof which relies on some more advanced 
matrix analysis 
(see \cite{Higham} and \cite{HJ1,HJ2} for background material). 
(In fact, the reviewer suggested this will work for general operators
using results from \cite{Conway}.)
Note that $S$ and $T$ are symmetric and positive semidefinite. 
We have 
$P_{\ran S}S = S = SS^\dagger S = SP_{\ran S^*} = SP_{\ran S}$. 
For the statement on the canonical polar decomposition, see
\cite[Theorem~8.3 and remarks on page~195]{Higham} 
which also has 
\begin{align}
U^*U = SS^\dagger = P_{\ran S}. 
\end{align}
This implies 
$LL^*U = (US)(US)^*U = USSU^*U = US^2P_{\ran S}=US^2=UL^*L$ and so 
\begin{equation}
\label{e:240929b}
T^2U = (LL^*)U = U(L^*L) = US^2
\end{equation}
and similarly
\begin{equation}
\label{e:240929c}
S^2U^* = (L^*L)U^* = U^*(LL^*) = U^*T^2. 
\end{equation}
The last identity extends to monomials in the form $(L^*L)^kU^* = U^*(LL^*)^k$ and 
further to polynomials. For $f$ suitably 
defined on the spectra of $LL^*$ and $L^*L$, 
we have, using \cite[Theorem~1.33]{Higham}, 
\begin{equation}
\label{e:241001a}
f(L^*L)U^* = U^*f(LL^*)
\end{equation}
and similarly $f(LL^*)U = Uf(L^*L)$. 

The spectra of $S$ and $T$ lie in $[0,1]$ which makes $A$ and $B$ well defined, 
with spectra in $[0,\pi/2]$. It follows that 
the spectra of $A/2$ and $B/2$ lie in $[0,\pi/4]$ which yields 
spectra of $\cos(A/2)$ and $\cos(B/2)$ in $[1/\sqrt{2},1]$. 
Hence $\cos(A/2)$ and $\cos(B/2)$ are invertible. 

Now set 
\begin{equation}
W := \begin{bmatrix}
\cos(A/2) & -U^*\sin(B/2)\\[+2mm]
-U\sin(A/2) & \cos(B/2)
\end{bmatrix}.
\end{equation}
The matrices $A$ and $B$ are symmetric, and so is $W$ 
because $(U\sin(A/2))^* = \sin(A/2)U^* = U^*\sin(B/2)$ 
follows from \cref{e:241001a} with $f(t)=\sin(\thalb\arcsin(\sqrt{t}))$. 
Moreover, since 
$\cos(A/2)\succ 0$, 
$\sin(A/2) = \frac{1}{2}\sin(A)(\cos(A/2))^{-1}$
and 
$P_{\ran S} = U^*U$, we have  
$\ran\sin(A/2)\subseteq\ran\sin(A) = \ran(S)=\Fix P_{\ran S}=
\Fix U^*U$. 
Because $\cos(B/2)\succ 0$, we obtain altogether
\begin{subequations}
\begin{align}
&\hspace{-1cm}\cos(A/2)-U^*\sin(B/2)\big(\cos(B/2)\big)^{-1}U\sin(A/2)\\
&= 
\cos(A/2)-U^*\big(\cos(B/2)\big)^{-1}\sin(B/2)U\sin(A/2)
\\
&=
\cos(A/2)-\big(\cos(A/2)\big)^{-1}U^*U\sin(A/2)\sin(A/2)
\\
&=
\cos(A/2)-\big(\cos(A/2)\big)^{-1}\sin^2(A/2)
\\
&=\big(\cos(A/2)\big)^{-1}\Big(\cos^2(A/2)-\sin^2(A/2) \Big)
\\
&=\big(\cos(A/2)\big)^{-1}\cos(A)
\\
&\succeq 0,
\end{align}
\end{subequations}
and a Schur complement argument shows that $W$ is positive semidefinite. 
We now show that $W^2=M$. 
We start with the $(W^2)_{1,1}$, the upper left block of $W^2$:
\begin{subequations}
\begin{align}
(W^2)_{1,1}
&= \cos^2(A/2) + U^*\sin(B/2)U\sin(A/2)\\
&= \cos^2(A/2) + \sin(A/2)U^*U\sin(A/2)\\
&= \cos^2(A/2) + \sin(A/2)\sin(A/2)\\
&=\Id_X.
\end{align}
\end{subequations}
Similarly, $(W^2)_{2,2} = \Id_Y$. 
Next,
\begin{subequations}
\begin{align}
(W^2)_{2,1}
&=
-U\sin(A/2)\cos(A/2)+\cos(B/2)\big(-U\sin(A/2)\big)
\\
&=
-U\sin(A/2)\cos(A/2)-U\cos(A/2)\sin(A/2)\\
&=-U\sin(A)\\
&=-L.
\end{align}
\end{subequations}
Similarly, $(W^2)_{1,2} = -L^*$.
We have shown that $W^2=M$, and so \cref{e:wowtrig} is verified. 
Finally, for $\theta\in[-1,1]$ we have 
\begin{subequations}
\begin{align}
\cos\big(\thalb\arcsin(\theta) \big)
&= \frac{\sqrt{1+\theta}+\sqrt{1-\theta}}{2}, \\
\sin\big(\thalb\arcsin(\theta) \big)
&= \frac{\sqrt{1+\theta}-\sqrt{1-\theta}}{2}, 
\end{align}
\end{subequations}
and this gives \cref{e:wowroot}.
\end{proof}

\begin{corollary}
Suppose that $X=\RR^n$, $Y=\RR^m$, and 
$L\in\RR^{m\times n}$ satisfies $\sigma^2\|L\|\leq 1$, where 
$\sigma>0$. 
Set $U := L(\sqrt{L^*L})^\dagger$. 
Then the principal square root of 
\begin{equation}
M := \begin{bmatrix}
\tfrac{1}{\sigma}\Id_X & -L^*\\
-L & \tfrac{1}{\sigma}\Id_Y
\end{bmatrix}
\end{equation}
is 
\begin{equation}
\frac{1}{2}
\begin{bmatrix}
\sqrt{\tfrac{1}{\sigma}\Id_X-\sqrt{L^*L}} + \sqrt{\tfrac{1}{\sigma}\Id_X+\sqrt{L^*L}} 
   &\;\; U^*\bigg(\sqrt{\tfrac{1}{\sigma}\Id_Y-\sqrt{LL^*}} - \sqrt{\tfrac{1}{\sigma}\Id_Y+\sqrt{LL^*}} \bigg)
\\[+4mm]
U\bigg(\sqrt{\tfrac{1}{\sigma}\Id_X-\sqrt{L^*L}} - \sqrt{\tfrac{1}{\sigma}\Id_X+\sqrt{L^*L}} \bigg) 
   &\;\; \sqrt{\tfrac{1}{\sigma}\Id_Y-\sqrt{LL^*}} + \sqrt{\tfrac{1}{\sigma}\Id_Y+\sqrt{LL^*}} 
\end{bmatrix}. 
\end{equation}
\end{corollary}
\begin{proof}
Set $S := \sqrt{L^*L}$, 
$T := \sqrt{LL^*}$, and 
$L_1 := \sigma L$. 
Then $\|L_1\|\leq 1$. 
Next, set ${S}_1:=
\sqrt{{L_1^*}{L_1}} =  \sigma\sqrt{L^*L}=\sigma S$, 
$U_1 := {L_1}({S_1})^\dagger = LS^\dagger = U$, and 
$T_1 := \sqrt{{L_1}{L_1^*}} =  \sigma\sqrt{LL^*}=\sigma T$. 
By \cref{l:240929}, the canonical polar decomposition of $L_1$ is
$L_1 = U_1S_1 = U(\sigma S)$ and we have the principal square root
of 
\begin{equation}
M_1 := \begin{bmatrix}
\Id_X & -L_1^*\\
-L_1 & \Id_Y
\end{bmatrix} = {\sigma} M
\end{equation}
is 
\begin{equation}
\sqrt{M_1}
= \frac{1}{2}
\begin{bmatrix}
\sqrt{\Id_X-S_1} + \sqrt{\Id_X+S_1} 
   &\;\; U^*\big(\sqrt{\Id_Y-T_1} - \sqrt{\Id_Y+T_1} \big)
\\[+4mm]
U\big(\sqrt{\Id_X-S_1} - \sqrt{\Id_X+S_1} \big) 
   &\;\; \sqrt{\Id_Y-T_1} + \sqrt{\Id_Y+T_1} 
\end{bmatrix}. 
\end{equation}
It follows that $\tfrac{1}{\sqrt{\sigma}}\sqrt{M_1} =\sqrt{M}$, i.e., 
\begin{equation}
\sqrt{M}
= \frac{1}{2\sqrt{\sigma}}
\begin{bmatrix}
\sqrt{\Id_X-\sigma S} + \sqrt{\Id_X+\sigma S} 
   &\;\; U^*\big(\sqrt{\Id_Y-\sigma T} - \sqrt{\Id_Y+\sigma T} \big)
\\[+4mm]
U\big(\sqrt{\Id_X-\sigma S} - \sqrt{\Id_X+\sigma S} \big) 
   &\;\; \sqrt{\Id_Y-\sigma T} + \sqrt{\Id_Y+\sigma T} 
\end{bmatrix}, 
\end{equation}
which yields the conclusion. 
\end{proof}

\subsubsection{An example where $\ran M$ is not closed}

\label{sss:notclosed}

Bredies et al.\ hinted in \cite[Remark~3.1]{BCLN}
the existence of $M$ which may fail
to have closed range. We now provide a setting
where $M$, $D$, and $C$ are explicit and where indeed 
$D$ is not smaller than $H$. 
To this end, we assume that 
$X=Y=\ell^2$ and 
\begin{equation}
L\colon \ell^2\to\ell^2\colon \bx = (x_n)_\nnn
\mapsto(\lambda_nx_n)_\nnn, \quad
\text{where 
$(\lambda_n)_\nnn$
is a sequence in $\left]0,1\right[$.}
\end{equation}
It is clear that $L$ is a one-to-one continuous linear operator 
with $\|L\| = \sup_\nnn\lambda_n \leq 1$, 
and $\cran L = Y$. Moreover, 
$L$ is onto $\Leftrightarrow$ $\inf_\nnn\lambda_n>0$.
We think of $L$ as a Cartesian product of 
operators $\RR\to\RR\colon \xi\mapsto \lambda_n\xi$.
Now assume that $\sigma=\tau=1$. 
In view of our work in \cref{ss:moreCPRR}, 
we can interpret the associated 
preconditioner $M$ with 
\begin{equation}
M \colon (x_0,x_1,\ldots) \mapsto  M_{\lambda_0}x_0\times M_{\lambda_1}x_1\times \cdots, 
\end{equation}
where $M_\lambda=C_\lambda C_\lambda^*$, using
\cref{e:230623a} and \cref{l:230629a}. 
Set $C\colon (x_0,x_1,\ldots)
\mapsto C_{\lambda_0}x_0\times C_{\lambda_1}x_1\times \cdots$.
Then $H=D=\cran C^*$.
Note that $M$ is a continuous linear operator, 
with 
\begin{equation}
\|M\| = \sup_\nnn \|M_{\lambda_n}\| = 1 + \sup_\nnn\lambda_n. 
\end{equation}
using \cref{e:230623b}. 
It is clear that $M$ is injective and has dense range because
each $M_{\lambda_n}$ is surjective. 
However, the smallest eigenvalue of $M_{\lambda_n}$
is $1-\lambda_n$; thus, 
\begin{equation}\label{Mbijectiveiff}
\text{$M$ is bijective}\;\;\Leftrightarrow\;\;
\sup_\nnn\lambda_n<1, 
\end{equation}
in which case $\|M^{-1}\| = (1-\sup_\nnn\lambda_n)^{-1}$ by 
\cref{e:230623c}. As $\ran M$ is dense, we have $\cran M=X$. It follows from \cref{Mbijectiveiff} that $\ran M=X\Leftrightarrow \sup_{\nnn} \lambda_n<1$.

It is time for a summary. 
\begin{proposition}
Suppose that $\sup_\nnn\lambda_n=1$. 
Then $M$ is not surjective and $\ran M$ is dense, but not closed. 
If $C\colon D\to H$ is any factorization of $M$
in the sense that $M=CC^*$, then $\ran C^*$ is not closed either. 
\end{proposition}
\begin{proof}
We observed already the statement on $M$.
If $\ran C^*$ was closed, then 
\cite[Lemma~8.40]{Deutsch} would imply that 
$\ran(C)=\ran(CC^*)=\ran(M)$ is closed which is absurd. 
\end{proof}

\begin{remark}
Suppose that $\sup_\nnn\lambda_n=1$. 
Then \cite[Proof of Proposition~2.3]{BCLN}
suggests to find $C$ by first computing the 
square root $S$ of $M$. Indeed, 
using \cref{e:230623d}, we see that $S$ is given by 
\begin{equation}
S\colon (x_0,x_1,\ldots) \mapsto 
S_{\lambda_0}x_0\times S_{\lambda_1}x_1\times \cdots, 
\end{equation}
restricted to $\ell^2$.
Once again, we have that $\ran S$ is dense but not closed. Hence $C=S$ and $D=H$ in this case. 
\end{remark}

\section*{Acknowledgments}
\small
We thank the Editor-in-Chief, Dr.\ Jong-Shi Pang, for his 
guidance and support. 
We also would like to thank 
two anonymous reviewers for their 
unusually careful reading and constructive comments which helped
us to improve this manuscript significantly.
The research of the authors was partially supported by Discovery Grants
of the Natural Sciences and Engineering Research Council of
Canada.

\end{document}